\documentclass[10pt,oneside, a4paper,reqno]{amsart}
\usepackage[utf8]{inputenc}

\title[$\tau$-tilting and $\mathbf{g}$-tameness for posets and concealed algebras]{$\tau$-tilting finiteness and $\mathbf{g}$-tameness: Incidence algebras of posets and concealed algebras}

\usepackage[foot]{amsaddr}

\author[E. D. Børve]{Erlend D. Børve}
\address[E.D.B.]{Institut Fourier, Université Grenoble Alpes, 100 rue des mathématiques, 38610 Gières, France (Current address: Institut für Mathematik und Wissenschaftliches Rechnen, Universität Graz, Mozartgasse 14, 8010 Graz, Austria)}
\email[E.D.B.]{erlend.d.borve@gmail.com}
\thanks{E.D.B. acknowledges support from the French ANR grant CHARMS (ANR-19-CE40-0017-02).}
\thanks{J.F.G. acknowledges support from The Research Council of Norway (RCN 301375).}
\author[J. F. Grevstad]{Jacob Fjeld Grevstad$^{\ast}$}
\address[J.F.G.]{Department of Mathematical Sciences, NTNU, NO-7491 Trondheim, Norway}
\email[J.F.G.]{jacob.f.grevstad@ntnu.no}
\author[E. S. Rundsveen]{Endre S. Rundsveen}
\address[E.S.R.]{Department of Mathematical Sciences, NTNU, NO-7491 Trondheim, Norway}
\address{$^{\ast}$Corresponding author}
\email[E.S.R.]{endre.s.rundsveen@ntnu.no}
\thanks{E.S.R. is funded by internal funds.}
\date{}

\input{setup.sty}

\begin{document}
\begin{abstract}
    We prove that any $\tau$-tilting finite incidence algebra of a finite poset is representation-finite, and that any $\mathbf{g}$-tame incidence algebra of a finite simply connected poset is tame. As the converse of these assertions are known to hold, we obtain characterizations of $\tau$-tilting finite incidence algebras and $\mathbf{g}$-tame simply connected incidence algebras. Both results are proved using the theory of concealed algebras. The former will be deduced from the fact that tame concealed algebras are $\tau$-tilting infinite, and to prove the latter, we show that wild concealed algebras are not $\mathbf{g}$-tame. We conjecture that any incidence algebra of a finite poset is wild if and only if it is not $\mathbf{g}$-tame, and prove a result showing that there are relatively few possible counterexamples. In the appendix, we determine the representation type of a $\tau$-tilting reduction of a concealed algebra of hyperbolic type.
\end{abstract}
	\maketitle
 \vspace{-1em}
    \tableofcontents
    \newpage
\section{Introduction}
\introheaderOne{Historical notes and motivation}
\introheaderTwo{Representation type.}
Representation theory of finite-dimensional algebras is at its core a tool to better understand algebras. Even though it can be traced back further, its current state was formed by Gabriel \cite{Gab72,Gab73,Gab80} using his language of quivers. The quiver-theoretic manner of approaching algebras provides the means of translating questions of algebraic structures into questions of linear algebra. However, the main selling point was arguably Gabriel's classification of finite-dimensional hereditary algebras with finitely many indecomposable modules in terms of Dynkin diagrams.

The question of how many indecomposable modules an algebra admits, has historically been a driving factor in the study of algebras. The well known Brauer--Thrall conjectures (see e.g. \cite{Jan57,Rin80}) are directly tied to this question. With regards to this question, algebras fall into three main categories of increasing complexity: finite, tame and wild type. Drozd famously showed that any algebra is either tame or wild, but never both \cite{Dro80}. Hence, there is a dichotomy of algebras by representation type. 

Around the same time as Gabriel proved his aforementioned result, the \emph{Kyiv-school} studied and classified poset-representations of finite type \cite{NR72,Kle72,NR75}. This result was highly influential and valuable in classifications of more general algebras, as one could often reduce to the case of posets. The main example of interest for us is Loupias's classification of representation-finite incidence algebras of posets \cite{Lou75,Lou75fr}.  

\introheaderTwo{($\tau$-)tilting.} Inspired by the methods used in Gabriel's classification, a new object was introduced, namely tilting modules \cite{GP72,BGP73,APR79,BB80,HR82}. This gave a method of constructing a new algebra from another while carrying over known structure. It paved the way for a staggering amount of progress, and is still an important tool for modern representation theorists. It is indeed fundamental to the arguments in this paper.

In an effort to close gaps in the theory of tilting when the underlying algebra is not hereditary, Adachi, Iyama and Reiten \cite{AIR14} generalized it to that of $\tau$-tilting theory. Their approach took inspiration from the theory of cluster algebras \cite{FZ02i,FZ02ii,BFZ05,FZ07}, as well as cluster tilting \cite{BMR07}.
By having a bijection with functorially finite torsion classes and providing a module-theoretic interpretation of $2$-term silting objects, $\tau$-tilting was soon accepted as an interesting and worthwhile direction of study.

\introheaderTwo{\g-representation type: $\tau$-tilting finiteness and $\g$-tameness.}
In a similar fashion as indecomposable modules, the amount of $\tau$-rigid indecomposable modules over an algebra has been closely studied. In fact, for hereditary algebras this amount is directly related to the algebras representation type. We say that an algebra is \textit{$\tau$-tilting finite} (or \textit{\g-finite}) if it only admits a finite number of $\tau$-rigid indecomposable modules \cite{DIJ19}. One easily observes that an infinite directed component in the AR-quiver provides infinitely many $\tau$-rigid modules, whence hereditary algebras of infinite representation type are $\tau$-tilting infinite, since they admit two such components (one postprojective and one preinjective).

Further, we say that an algebra is \textit{$\g$-tame} if the $\g$-vector fan is dense in the real Grothendieck group \cite{AY23}. It has been shown that tame representation type is sufficient for $\g$-tameness \cite{PY23}, and for hereditary algebras it \revised{is} in fact also necessary \cite{Hil06}. We do however have examples of $\g$-tame algebras of wild representation type (for instance wild local algebras and certain Jacobian algebras \cite{Yur20}). At the time of writing, the authors are not familiar with a definition of $\g$-wildness of an algebra, apart from simply negating the notion of $\g$-tameness. 

Having, for the hereditary case, two equivalent notions to finite and tame representation type, it is natural to ask how big of a defect we might get for less well-behaved algebras. We look at the slightly more rowdy class of incidence algebras in this paper. Both the classification of minimally representation-infinite incidence algebras by Loupias \cite{Lou75,Lou75fr} and the criterion of tameness for incidence algebras by Leszczyński \cite{Les03} will be prominent in the following.

\introheaderOne{Original results}

\introheaderTwo{$\tau$-tilting finiteness.} Following the motivation outlined above we first set forth on investigating when incidence algebras are $\tau$-tilting finite. The strategy for our investigation relies on three observations. The first being that representation-infinite concealed $\K$-algebras are $\tau$-tilting infinite, see \Cref{lem:concealed_tau_inf}. The second being that the twelve classes of minimal representation infinite incidence algebras provided by Loupias (see \Cref{tab:Loupias frames} on page \pageref{tab:Loupias frames}) are tame concealed, see \Cref{cor:Lou75concealed}. The third and final being that the reduction techniques of Loupias preserves $\tau$-tilting finiteness, see \Cref{lem:Pla19}. 
With this strategy we prove our first main result.

\begin{reptheorem}{thm:tautfiniteposets}
    Let $\K$ be a field and let $(P,\leq)$ be a finite poset. The incidence algebra $\incidencealgebra{\K}{P}$ is $\tau$-tilting finite if and only if it is representation-finite.
\end{reptheorem}

\introheaderTwo{$\g$-tameness.} Inspired by the situation for hereditary algebras and the fact that $\tau$-tilting finite incidence algebras are representation finite, we wondered whether $\g$-tame incidence algebras were tame as well. When the underlying poset is simply connected, we answer this in the affirmative. In the multiply connected case we conjecture that the same will be true, based on some worked examples.

Mirroring the strategy for our first result we develop a means of reduction here as well. We begin by slightly expanding a result of Happel and Unger \cite{HU89} telling us that for a wild concealed algebra we can find a suitable factor algebra which is concealed of \emph{hyperbolic type}, see \Cref{definition:hyperbolic}. 

\begin{reptheorem}{thm:red_to_hyperbolic}
    Let $\K$ be an algebraically closed field.
    If $B$ is a concealed $\K$-algebra of wild type, then there exists an idempotent $e\in B$ such that $B/(e)$ is concealed of hyperbolic type.
\end{reptheorem}

This is significant for us since $\g$-tameness is inherited when passing to factor algebras by idempotent ideals \cite[Proposition 3.11]{PY23} (see \Cref{lem:convidem}\eqref{lem.PY23.3.11}). Following this we show that for hyperbolic hereditary algebras $H$ we can carry information from the $\g$-vector fan over to their concealed algebras, \Cref{lem:negative_cone_and_postprojectives}. This, together with Asai's description of $\g$-vector fans through the wall and chamber structure of $K_0(\proj \Lambda)$ \cite{BT19, Asa21}, allows us to prove the following novel result. 

\begin{reptheorem}{thm:notgtame}
Let $\K$ be an algebraically closed field and let $B$ be a wild concealed $\K$-algebra. 
Then $B$ is not $\mathbf{g}$-tame.
\end{reptheorem}

This result is noteworthy in itself, but for us it is mainly in connection with Leszczyński's classification of wild incidence algebras \cite[Theorem~1.4]{Les03} where it really shines. Specifically, it leads to the following corollary of \Cref{thm:notgtame}.

\begin{repcorollary}{cor:concealed_gvecfan}
    Suppose that the field $\K$ is algebraically closed.
    Let $(P,\leq)$ be a finite simply connected poset. Then the incidence $\K$-algebra $\incidencealgebra{\K}{P}$ is $\mathbf{g}$-tame if and only if it is tame.
\end{repcorollary}

In \Cref{conj:g_tame_poset}, we postulate that this result holds for any incidence algebra of a finite poset. To justfy that our conjecture is a reasonable one, we include \Cref{prop:conj_lim}, where we argue that there are relatively few counterexamples, if any at all.

\introheaderTwo{$\tau$-tilting reduction.} At the end of the paper we include an application of our results. That is, in \Cref{sec:red}, we use \Cref{thm:red_to_hyperbolic} to prove results on the $\tau$-tilting reduction \cite{Jas15} of concealed algebras. Our main result in this appendix will be the following.

 \begin{reptheorem}{prop:MinimalLemJassoRed}
    Let $\K$ be an algebraically closed field, let $Q$ be a hyperbolic quiver and let $B$ be a concealed $\K$-algebra of type $Q$.
    For $(M,\,R)\in\taurigidpair{B}$\revised{, where $M$ is postprojective, }the following assertions hold:
    \begin{enumerate}
    \setcounter{enumi}{1}
        \item If $|M|+|R|\geq 1$, then the $\tau$-tilting reduction of $B$ with respect to $(M,\,R)$ is tame.
        \item If $|M|+|R|\geq 2$, then the $\tau$-tilting reduction of $B$ with respect to $(M,\,R)$ is representation-finite.
    \end{enumerate}
\end{reptheorem}

\introheaderOne{Notation and conventions.} A field $\K$ is fixed throughout. All modules over a finite-dimensional $\K$-algebra will be left modules, as opposed to right modules. For a finite-dimensional $\K$-algebra $\Lambda$, let $\Mod{\Lambda}$ denote the category of left $\Lambda$-modules, and let $\Modf(\Lambda)$ denote the \revised{full} subcategory of finite-dimensional $\Lambda$-modules. Arrows in a quiver are composed from right to left. A \textit{$\K$-linear category} is an enriched category over $\Mod{\K}$. It need not admit direct sums.

\subsection*{Acknowledgements.} We are grateful to Aslak Buan, Eric Hanson, Kaveh Mousavand, Steffen Oppermann, Pierre-Guy Plamondon and Håvard Terland for insightful discussions. \revised{We thank an anonymous reviewer for helpful comments.} E.D.B acknowledges support from the French ANR grant CHARMS (ANR-19-CE400017). J.F.G. acknowledges support from The Research Council of Norway (RCN 301375). Some of our diagrams were made with the help of \url{https://q.uiver.app/}.

\section{Generalities on incidence algebras of posets}
A \textit{partially ordered set}, or \textit{poset} for short, is a tuple $(P,\leq)$, where $P$ is a set and $\leq$ is a binary relation on $P$ which is reflexive, transitive and anti-symmetric. When the binary relation $\leq$ is clear from context, or if the poset in question is arbitrary, we may omit $\leq$ from the notation, so a reference to the underlying set $P$ is meant as reference to the entire structure. 
An \textit{interval} in a poset $P$ is a subposet of the form
\begin{equation*}
    [a,b] \defeq \{x \sth  a\leq x\leq b\} \subseteq P,
\end{equation*}
where $a,b\in P$. We denote the set of non-empty intervals in $P$ by $\intervals{P}$.
A poset is \textit{locally finite} if every interval therein is finite. 

\begin{definition}[{\cite[§3]{Rot64}}]
Let $(P,\leq)$ be a locally finite poset. Denote the set of functions from $\intervals{P}$ to $\K$ by $\K^{\intervals{P}}$, and endow this set with a $\K$-vector space structure using pointwise operations.
The \textit{incidence $\K$-algebra} of $P$ has $\K^{\intervals{P}}$ as the underlying $\K$-vector space, and multiplication
\begin{equation*}
    \begin{tikzcd}
        \K^{\intervals{P}} \otimes_{\K} \K^{\intervals{P}} \arrow[r,"\ast"] & \K^{\intervals{P}} 
    \end{tikzcd}
\end{equation*}
given by $(f\ast g)([a,b]) \defeq \sum\limits_{t\in [a,b]} f([t,b])g([a,t])$. We denote the incidence $\K$-algebra of $P$ by $\incidencealgebra{\K}{P}$.
\end{definition}
It is easy to see that the incidence algebra of $P$ is finite-dimensional precisely when $P$ is finite. Indeed, it is equivalent to the assertion that $\intervals{P}$ is a finite set.
In this case, one can present the $\K$-algebra $\incidencealgebra{\K}{P}$ as the quotient of the free $\K$-algebra generated by $\intervals{P}$ modulo the relations $\{[c,d][a,b] - \delta_{b,c}[a,d] \}$, where $\delta_{\mhyphen,\mhyphen}$ is the Kronecker delta symbol. In the following, all posets we consider will be finite.

We consistently refer to oriented multigraphs as \textit{quivers}. Recall that we compose paths in a quiver from right to left. The \textit{Hasse quiver} of a poset $P$ is the quiver where the vertices are the elements in $P$, and we have an arrow $x\rightarrow y$ if $x\leq y$ and no element lies strictly between $x$ and $y$. We will denote this quiver by $\HasseQuiver{P}$. Letting $I_P$ denote the ideal of the path $\K$-algebra $\K \HasseQuiver{P}$ generated by the relations $\{p_1-p_2\}_{p_1,p_2}$, where we index over all pairs of paths with the same start and endpoints, one easily proves that $\incidencealgebra{\K}{P}$ is isomorphic to the bound path $\K$-algebra $\K \HasseQuiver{P} / I_P$. As a consequence, a left module over the incidence algebra $\incidencealgebra{\K}{P}$ can be regarded as a covariant representation of the bound quiver $(\HasseQuiver{P},I_P)$ \cite[Lemma~2.7]{Lad08}. Finite-dimensional representations of $(Q_P,I_P)$ correspond exactly to finite-dimensional $\incidencealgebra{\K}{P}$-modules. 

\begin{remark}\label{rem:posetreps}
    We will refrain from referring to $\incidencealgebra{\K}{P}$-modules as ``representations of $P$.''
    For many authors, a \textit{representation} of a poset $P$ in a $\K$-vector space $V$ amounts to an order-preserving map from $P$ into the poset of subspaces of $V$ 
    \cite[§8]{Rot64} \cite[§1]{Dro74} \cite{Naz74}. The term \textit{$P$-space} has also been used for this notion of representation. Equivalently, they can be regarded as socle-projective modules of the incidence $\K$-algebra of $P^{\star}$, where $P^{\star}$ is obtained from $P$ by adding a maximal element. The maps in such representations of the bound quiver $(\HasseQuiver{P},I_P)$ are all injective, which we, in contrast, will not require. 
\end{remark}

\begin{example}
    Let $n\geq 1$. A poset is of \textit{type $\qA_n$} if it has $n$ elements and the underlying unoriented graph of its Hasse quiver takes the form
    \begin{equation}\label{eq:An}\tag{$\qA_n$}
    \begin{tikzcd}
        \bullet \arrow[r,no head] & \bullet \arrow[r,no head] & \cdots \arrow[r,no head] & \bullet.
    \end{tikzcd}
    \end{equation}
    The incidence $\K$-algebra of such a poset is isomorphic to the path algebra of the correspondingly oriented $\qA_n$-quiver. Examples include the totally ordered sets, whose Hasser quivers are given by
\begin{equation}\label{eq:n}\tag{$\totallyordered{n}$}
    \begin{tikzcd}[every arrow/.append style={-latex}]
        1 \arrow[r] & 2 \arrow[r] & \cdots \arrow[r] & n.
    \end{tikzcd}
    \end{equation}
    Their incidence $\K$-algebras are isomorphic to path $\K$-algebras of linearly ordered $\qA_n$-quivers. 
\end{example}

Given a connected bound quiver $(Q,I)$, let $N(Q,I)$ be the classifying space of the category defined as the quotient of the free category of $Q$, modulo the ideal identifying two paths $p_1$ and $p_2$ if there exists a minimal relation in $I$ of the form $\sum_{p\in L} \lambda_p p$, where $p_1,p_2\in L$. If the bound quiver $(Q_P,I_P)$ is the bound Hasse quiver of a poset $P$, then $N(Q_P,I_P)$ is simply the classifying space of $P$. A bound quiver $(Q,I)$ is \textit{connected} if $N(Q,I)$ is connected as a topological space. 
We define the \textit{fundamental groupoid} $\Pi_1(Q,I)$ of $(Q,I)$ as the fundamental groupoid of $N(Q,I)$. A \textit{connected component} of $Q$ can be defined as a component of $\Pi_1(Q,I)$. One says that $(Q,I)$ is \textit{simply connected} (resp. \textit{multiply connected}) if it is connected and the fundamental groupoid $\Pi_1(Q,I)$ is equivalent to the trivial groupoid (resp. not equivalent to the trivial groupoid). We say that a finite poset $P$ is \textit{connected} (resp. \textit{simply connected}, resp. \textit{multiply connected}) if the bound quiver $(Q_P,I_P)$ is {connected} (resp. {simply connected}, resp. {multiply connected}). Posets admitting a unique maximal element, or a unique minimal element, are simply connected.

\section{\texorpdfstring{$\tau$}{tau}-tilting finiteness of incidence algebras}

This section is concerned with $\tau$-tilting finiteness of incidence algebras. We begin by recalling basic definitions of $\tau$-tilting theory and concealed algebras in \Cref{sec:2.2tau tilting} and \Cref{sec:2.2.5concealed algebras}. The main result of this section is the following, which will be shown in \Cref{sec:2.3tau tilting finite incidence}.
\begin{reptheorem}{thm:tautfiniteposets}
    Let $P$ be a finite poset. The incidence $\K$-algebra $\incidencealgebra{\K}{P}$ is $\tau$-tilting finite if and only if it is representation-finite.
\end{reptheorem}  

\subsection{\texorpdfstring{$\tau$}{tau}-tilting theory}\label{sec:2.2tau tilting}
$\tau$-tilting theory was developed by Adachi--Iyama--Reiten \cite{AIR14}. We recall the elementary definitions. 

Let $\Lambda$ be a finite-dimensional $\K$-algebra. Recall that a $\Lambda$-module is \textit{basic} if no indecomposable direct summand is repeated in its decomposition into indecomposable direct summands. A pair $(M,P)$ of $\Lambda$-modules is \textit{basic} if both $M$ and $P$ are basic as $\Lambda$-modules.
\begin{definition}\label{def:tautilt}
Let $\Lambda$ be a finite-dimensional algebra over a field $\K$.
\begin{enumerate}
    \item\label{def:tautilt.rigid} A finitely generated left $\Lambda$-module $M$ is said to be \textit{rigid} if $\Ext_{\Lambda}^{1}({M},\,{M})=0$.
    \item A finitely generated left $\Lambda$-module $M$ is said to be \textit{$\tau$-rigid} if $\Hom_{\Lambda}({M},\,{\tau M})=0$, where $\tau$ is the Auslander--Reiten translation of $\Lambda$. The set of isomorphism classes of basic $\tau$-rigid $\Lambda$-modules is denoted by $\taurigid{\Lambda}$.
    \item A pair $(M,\,P)$ is \textit{support $\tau$-rigid} if $M$ is a $\tau$-rigid module and $P$ is a finitely generated projective left $\Lambda$-module such that $\Hom_{\Lambda}({P},\,{M})=0$. The set of isomorphism classes of basic $\tau$-rigid pairs in $\Modf(\Lambda)$ is denoted by $\taurigidpair{\Lambda}$.
    \item A support $\tau$-rigid pair $(M,\,P)$ is \textit{support $\tau$-tilting} if $|M|+|P|=|\Lambda|$, where $|X|$ is the number of non-isomorphic indecomposable direct summands of a left $\Lambda$-module $X$. The set of isomorphism classes of basic support $\tau$-tilting pairs in $\Modf(\Lambda)$ is denoted by $\stautiltpair{\Lambda}$.
    \item A left $\Lambda$-module $M$ is \textit{support $\tau$-tilting} if there exists a projective left $\Lambda$-module $P$ such that $(M,\,P)$ is a support $\tau$-tilting pair. The set of isomorphism classes of basic support $\tau$-tilting $\Lambda$-modules is denoted by $\stautilt{\Lambda}$.
    \item A support $\tau$-rigid pair $(N,\,Q)$ is a \textit{direct summand} of the support $\tau$-rigid pair $(M,\,P)$ if $N$ is a direct summand of $M$ and $Q$ is a direct summand of $P$. 
\end{enumerate}
\end{definition} 

\begin{definition}[{\cite[Definition 1.1]{DIJ19}}]\label{def:tautiltfinite}
    We say that a finite-dimensional algebra $\Lambda$ is \textit{$\tau$-tilting finite} if the set $\stautilt{\Lambda}$ is finite (or equivalently that $\taurigid{\Lambda}$ is finite). If $\Lambda$ is not {$\tau$-tilting finite}, it is said to be \textit{$\tau$-tilting infinite}.
\end{definition}

\subsection{Tilted and concealed algebras}\label{sec:2.2.5concealed algebras}
For a finite-dimensional $\K$-algebra $\Lambda$ and a finite-dimensional $\Lambda$-module $M$, let $\add(M)$ denote the smallest additive subcategory of $\Modf(\Lambda)$ \revised{which is closed under direct summands and} \revised{contains} $M$.

\begin{definition}
    Let $\Lambda$ be a finite-dimensional $\K$-algebra.
   We say that $T\in\Modf(\Lambda)$ is a \textit{tilting $\Lambda$-module} if the following hold:
    \begin{enumerate}
        \item The $\Lambda$-module $T$ is rigid.
        \item There exists a short-exact sequence in $\Modf(\Lambda)$
        \begin{equation*}
            \begin{tikzcd}
                0 \arrow[r] & P^{T}_1  \arrow[r] & P^{T}_0  \arrow[r,] & T  \arrow[r] & 0
            \end{tikzcd}
        \end{equation*}
        where the terms $P^{T}_1$ and $P^{T}_0$ are finite-dimensional projective $\Lambda$-modules.
        \item There exists a short-exact sequence in $\Modf(\Lambda)$
        \begin{equation*}
            \begin{tikzcd}
                0 \arrow[r] & \Lambda  \arrow[r] & {T}_0  \arrow[r] & T_1  \arrow[r] & 0
            \end{tikzcd}
        \end{equation*}
        where the terms $T_0$ and $T_1$ are in $\add(T)$. 
    \end{enumerate}
\end{definition}

Let $Q$ be an acyclic quiver. Recall that if the path $\K$-algebra $\K Q$ is representation-infinite, the set of indecomposable $\K Q$-modules can be partitioned into three distinct classes:
\begin{itemize}
    \item the \textit{postprojective $\K Q$-modules}, namely the $\K Q$-modules $M$ for which there exists an indecomposable finite-dimensional projective $\K Q$-module $P$ and a non-negative integer $i$ such that $M\simeq \tau^{-i}P$.
    \item the \textit{preinjective $\K Q$-modules}, namely the $\K Q$-modules $M$ for which there exists an indecomposable finite-dimensional injective $\K Q$-module $I$ and a non-negative integer $i$ such that $M\simeq \tau^{i}I$.
    \item the \textit{regular $\K Q$-modules}, that are neither postprojective nor preinjective. 
\end{itemize}

\begin{definition}
    Let $\Lambda$ be a finite-dimensional $\K$-algebra.
    \begin{enumerate}
        \item We say that $\Lambda$ \textit{has a postprojective component} if the Auslander--Reiten quiver of $\Modf(\Lambda)$ has an acyclic connected component in which every indecomposable module is isomorphic to a module of the form $\tau^{-i}P$ for some indecomposable finite-dimensional projective $\Lambda$-module $P$ and some non-negative integer $i$.
        \item We say that $\Lambda$ \textit{has a preinjective component} if the Auslander--Reiten quiver of $\Modf(\Lambda)$ has an acyclic connected component in which every indecomposable module is isomorphic to a module of the form $\tau^{i}I$ for some indecomposable finite-dimensional injective $\Lambda$-module $I$ and some non-negative integer $i$.
    \end{enumerate}
\end{definition}

Path $\K$-algebras of quivers admit postprojective and preinjective components.
For each connected component of $Q$ there is a postprojective component in the AR-quiver of $\K Q$, and likewise for the preinjectives. Such a connected component $Q'$ of $Q$ yields an infinite postprojective (resp. preinjective) component if and only if $Q'$ is a representation-infinite quiver.

\begin{definition}
    Let $Q$ be an acyclic quiver. 
    \begin{enumerate}
        \item  A \textit{tilted $\K$-algebra of type $Q$} is isomorphic to $\End_{\K Q}(T)\op$ for some tilting $\K Q$-module $T$.
        \item If $T$ is a postprojective tilting $\K Q$-module, we say that the tilted algebra (and any $\K$-algebra to which it is isomorphic) is a \textit{concealed $\K$-algebra of type $Q$}. 
        \item Let $B = \End_{\K Q}(T)\op$ be a concealed $\K$-algebra of type $Q$. We say that $B$ is \textit{tame concealed} (resp. \textit{wild concealed}) if $Q$ is a tame and representation-infinite (resp. wild) quiver.
    \end{enumerate} 
\end{definition}

Let $Q$ be a representation-infinite acyclic quiver, and let $B$ be a concealed $\K$-algebra of type $Q$.
Then $B$ admits one postprojective component (resp. preinjective component) for every connected component of $Q$, and at least one which contains (countably) infinitely many indecomposable $B$-modules.

The following lemma is well known and immediately proved (cf. \cite[Remark~2.9]{Mou23}).
\begin{lemma}\label{lem:concealed_tau_inf}
    Let $B$ be a tame concealed or a wild concealed $\K$-algebra. Then $B$ is $\tau$-tilting infinite. 
\end{lemma}
\begin{proof}
    Tame concealed and wild concealed $\K$-algebras admit at least one infinite postprojective (and at least one infinite preinjective) component. 
    It follows easily from directedness that the indecomposable modules constituting postprojective or preinjective components are all $\tau$-rigid.
    Thus, the postprojective (or preinjective) component of $B$ provides a (countably) infinite set of $\tau$-rigid $B$-modules.
\end{proof}

Let $\Lambda$ be a finite dimensional algebra. For a full subcategory $\mathcal{X}$ of $\Modf(\Lambda)$, let $\Gen\mathcal{X}$ denote the smallest full subcategory of $\Modf(\Lambda)$ containing $\mathcal{X}$ which is closed under factor modules, and let $\mathcal{X}^{\perp}$ (resp. ${^{\perp}\mathcal{X}}$) denote the full subcategory of $\Modf(\Lambda)$ spanned by the $\Lambda$-modules $Y$ for which $\Hom_{\Lambda}(X,Y)=0$ (resp. $\Hom_{\Lambda}(Y,X)=0$) for all $X\in\mathcal{X}$. When $\mathcal{X}$ is of the form $\add(T)$ for some $\Lambda$-module $T$, we simplify the notation to $\Gen{T}$ and $T^{\perp}$, respectively.
Recall that a \textit{torsion class} (resp. \textit{torsion-free class}) of $\Modf(\Lambda)$ is a full subcategory of $\Modf(\Lambda)$ which is closed under factor modules (resp. submodules) and extensions. A \textit{torsion pair} in $\Modf(\Lambda)$ is a pair $(\mathcal{T},\,\mathcal{F})$ of full subcategories of $\Modf(\Lambda)$ such that $\mathcal{F} = \mathcal{T}^{\perp}$ and $\mathcal{T} = {^{\perp}\mathcal{F}}$. In a torsion pair $(\mathcal{T},\,\mathcal{F})$, we have that $\mathcal{T}$ is a torsion class and that $\mathcal{F}$ is a torsion-free class \cite{Dic66}. A torsion pair $(\mathcal{X},\,\mathcal{Y})$ in $\Modf(\Lambda)$ is \textit{split} if all indecomposable finite-dimensional $\Lambda$-modules are either in $\mathcal{X}$ or in $\mathcal{Y}$.

Tilting theory lets us compute the module categories of tilted $\K$-algebras. This is thanks to the Brenner--Butler Theorem, namely the following.
\begin{theorem}[Brenner--Butler equivalences, {\cite{BB80,Bon81,HR82}}]\label{setup:concealed}
    Let $Q$ be an acyclic quiver, and let $T$ be a tilting $\K Q$-module. Let $\Lambda=\End_{\K Q}(T)\op$ denote the {tilted algebra} obtained from $T$.
    There is a torsion pair $(\mathcal{T}_T,\,\mathcal{F}_T)\defeq (\Gen{T},\,T^{\perp})$ in $\mod{\K Q}$, and a split torsion pair $(\mathcal{X}_T,\,\mathcal{Y}_T)$ in $\Modf(\Lambda)$, such that the following functors are mutually inverse additive equivalences:
    \[
    \begin{tikzpicture}
        \node (A) at (0,0) []{$\cT_T$};
        \node (B) at (2,0) [] {$\cY_T$};
        \draw[->] (A) to[out=30,in=150] node[midway,above,scale=.7]{$\Hom_{\K Q}(T,\,-)$} (B);
        \draw[->] (B) to[out=210,in=-30] node[midway,below,scale=.7]{$T\otimes_{\Lambda}-$} (A);
        
        \node (a) at (3.5,0) [] {and};
        \node (C) at (5,0) [] {$\cF_T$};
        \node (D) at (7,0) [] {$\cX_T$.};
        \draw[->] (C) to[out=30,in=150] node[midway,above,scale=.7]{$\Ext_{\K Q}^1(T,\,-)$} (D);
        \draw[->] (D) to[out=210,in=-30] node[midway,below,scale=.7]{$\Tor_1^\Lambda(T,\,-)$} (C);
    \end{tikzpicture}
    \]
\end{theorem}

An immediate corollary of \Cref{setup:concealed} is the following: if $Q$ is a Dynkin quiver, then any tilted $\K$-algebra of type $Q$ is representation-finite.

\subsection{\texorpdfstring{$\tau$}{tau}-tilting finite incidence algebras are representation-finite} \label{sec:2.3tau tilting finite incidence}
Recall that a finite-dimensional $\K$-algebra $\Lambda$ is called \textit{representation-finite} if there are only finitely many isomorphism classes of indecomposable left $\Lambda$-modules.

\begin{remark}\label{rem:repfin_allfields}
For incidence algebras of posets, the choice of field does not matter when it comes to representation-finiteness. Indeed, the incidence $\K$-algebra $\incidencealgebra{\K}{P}$ of a poset $P$ is representation-finite for some field $\K$ precisely when it is representation-finite for all fields $\K$. This is shown by setting up a correspondence between posets with representation-finite incidence $\K$-algebras and hammocks \cite[Corollary 4]{RV87}, the latter being independent of the field $\K$.
\end{remark}

Loupias devises two reduction procedures of incidence algebras that preserve representation-finiteness \cite[Propositions 1.2 and 1.3]{Lou75}. 
Given a poset $P$, a subposet of $P$ is also called a \textit{subtraction} of $P$.
A poset $L$ is a \textit{contraction} of $P$ if there exists a surjective order-preserving map $P\rightarrow L$ with connected fibers. An \textit{elementary contraction} $c\colon P\to R$ is a contraction for which there exists exactly one element $x\in R$ such that $|c^{-1}(x)|=2$ and $|c^{-1}(r)|=1$ for all $r\in R\setminus \{x\}$. Informally, an elementary contraction identifies two adjacent elements in $P$. Thus, there are as many elementary contractions as there are arrows in the Hasse quiver of $P$. 
It can be shown that all contractions are composites of elementary contractions.

We say that a poset $P$ \textit{can be reduced to $L$} if it is the case that $L$ can be obtained from $P$ by subtraction and contraction in a finite number of steps. We may assume that each step either removes a single vertex or contracts an arrow in the Hasse quiver.

If $L$ is a subposet of $P$, one sets up a $\K$-algebra isomorphism 
\begin{equation*}
    \incidencealgebra{\K}{L} \simeq e\incidencealgebra{\K}{P}e,
\end{equation*}
where $e$ is the idempotent $\sum_{i \in L}e_i$ and the terms $e_i$ in the sum are primitive idempotents corresponding to the elements in $L$. Consequently, we have a fully faithful functor
\begin{equation}\label{eq:fullyfaithful_subtr}
    \Modf(\incidencealgebra{\K}{L})\simeq \Modf(e\incidencealgebra{\K}{P}e)\xhookrightarrow{\Hom_{e\Lambda e}({e\Lambda},\,{-})} \Modf(\incidencealgebra{\K}{P}).
\end{equation}
If $R$ is a contraction of $P$, we have, by definition, an epimorphism of posets $P\twoheadrightarrow R$. 
Since the module category $\Modf(\incidencealgebra{\K}{P})$ is equivalent to the functor category $\Modf(\K)^{P}$, and similarly $\Modf(\incidencealgebra{\K}{R})\cong \Modf(\K)^{R}$, one can induce a  functor
\begin{equation}\label{eq:fullyfaithful_contr}
   \Modf(\incidencealgebra{\K}{R})\longhookrightarrow\Modf(\incidencealgebra{\K}{P})
\end{equation}
directly from the map $P\twoheadrightarrow R$. Since this is an epimorphism with connected fibers, it follows that the functor \eqref{eq:fullyfaithful_contr} is fully faithful \cite[Proposition 8]{Lou75fr}. 
The existence of the fully faithful functors displayed in \eqref{eq:fullyfaithful_subtr} and \eqref{eq:fullyfaithful_contr} is all one needs in order to prove the following lemma.
\begin{lemma}[{\cite[Propositions 1.2 and 1.3]{Lou75}}]\label{lem:Lou75}
    Let $P$ be a poset such that the incidence $\K$-algebra $\incidencealgebra{\K}{P}$ is representation-finite.
    \begin{enumerate}
        \item\label{lem:Lou75.1} Any subposet of $P$ has a representation-finite incidence $\K$-algebra.
        \item\label{lem:Lou75.2} Any contraction of $P$ has a representation-finite incidence $\K$-algebra.
    \end{enumerate}
\end{lemma}

\begin{example}\label{ex:AnDntilde}
    Let $P$ be a poset of \emph{type \ref{Antilde}}, where $n\geq 3$, namely a poset with Hasse quiver as shown below, where the undirected arrows may go in either direction:
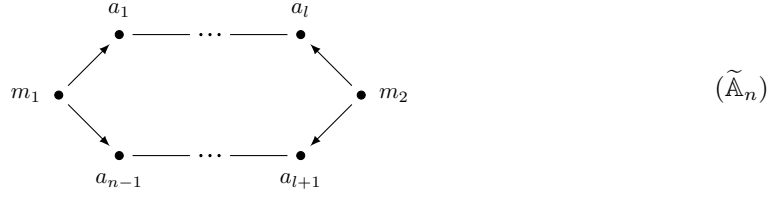
\begin{equation}\label{Antilde}\tag{$\qAtilde_n$}
    \begin{tikzpicture}[tips=proper,scale=.8,baseline=(current  bounding  box.center)]
        \node[nodeDots] (A) at (0,0)  {};
        \node[left=2pt of A] [scale=.85] {$m_1$};
        \node[nodeDots] (B) at (1,1) {};
        \node[above=2pt of B] [scale=.85] {$a_1$};
        \node (C) at (2.5,1) []{\tikzdots};
        \node[nodeDots] (D) at (4,1) {};
        \node[above=2pt of D] [scale=.85] {$a_l$};
        \node[nodeDots] (E) at (5,0) {};
        \node[right=2pt of E] [scale=.85] {$m_2$};
        \node[nodeDots] (F) at (4,-1) {};
        \node[below=2pt of F] [scale=.85] {$a_{l+1}$};
        \node (G) at (2.5,-1) []{\tikzdots};
        \node[nodeDots] (H) at (1,-1) {};
        \node[below=2pt of H] [scale=.85] {$a_{n-1}$};
        
        \draw[-latex, shorten <=5pt, shorten >=5pt] 
            {(A.center) edge (B.center)}
            {(A.center) edge (H.center)}
            {(E.center) edge (D.center)}
            {(E.center) edge (F.center)}
        ;
        \draw[shorten <=5pt, shorten >=5pt]
            {(B.center) edge[shorten >=0pt] (C.west)}
            {(C.east)   edge[shorten <=0pt] (D.center)}
            
            {(H.center) edge[shorten >=0pt] (G.west)}
            {(G.east)   edge[shorten <=0pt] (F.center)}
        ;
    \end{tikzpicture}
    \end{equation}
    As a matter of definition, posets of type \ref{Antilde} have circular Hasse quivers.
Note that there are at least two minimal elements, denoted here by $m_1$ and $m_2$. Such posets can be contracted onto a poset of type \ref{egA3tilde}, namely a poset of the following form:
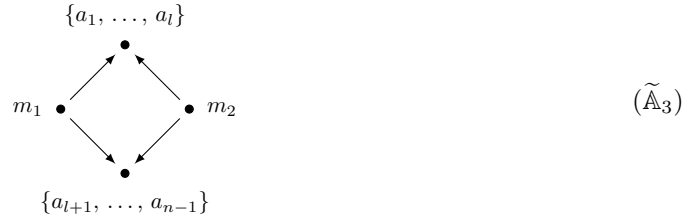
\begin{equation}\label{egA3tilde}\tag{$\qAtilde_3$}
    \begin{tikzpicture}[tips=proper,scale=.85, baseline=(current bounding  box.center)]
    \node[nodeDots] (A) at (0,0) {};
    \node[left=2pt of A] [scale=.85] {$m_1$};
    \node[nodeDots] (B) at (1,1) {};
    \node[above=2pt of B] [scale=.85] {$\{a_1,\,\ldots,\,a_l\}$};
    \node[nodeDots] (C) at (2,0) {};
    \node[right=2pt of C] [scale=.85] {$m_2$};
    \node[nodeDots] (D) at (1,-1) {};
    \node[below=2pt of D] [scale=.85] {$\{a_{l+1},\,\ldots,\,a_{n-1}\}$};

    \draw[-latex, shorten <=5pt, shorten >=5pt] 
        {(A.center) edge (B.center)}
        {(A.center) edge (D.center)}
        {(C.center) edge (B.center)}
        {(C.center) edge (D.center)}
    ;
\end{tikzpicture}
\end{equation}
Similarly, a poset of \emph{type \ref{Dntilde}}, whose underlying undirected graph is of the form
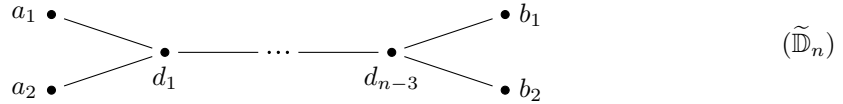
\begin{equation}\label{Dntilde}\tag{$\qDtilde_n$}
    \begin{tikzpicture}[ baseline=(current bounding  box.center)]
    \node[nodeDots, label=left:{$a_1$}] (A1) at (0,.5) {};
    \node[nodeDots, label=left:{$a_2$}] (A2) at (0,-.5) {};
    \node[nodeDots, label=below:{$d_1$}] (D1) at (1.5,0) {};
    \node (dots) at (3,0) [] {\tikzdots};
    \node[nodeDots, label=below:{$d_{n-3}$}] (Dn-3) at (4.5,0) {};
    \node[nodeDots, label=right:{$b_1$}] (B1) at (6,.5) {};
    \node[nodeDots, label=right:{$b_2$}] (B2) at (6,-.5) {};

    \draw[shorten <=5pt, shorten >=5pt]
        {(A1.center) edge (D1.center)}
        {(A2.center) edge (D1.center)}
        
        {(D1.center) edge[shorten >=0pt] (dots.west)}
        {(dots.east) edge[shorten <=0pt] (Dn-3.center)}
    
        {(Dn-3.center) edge (B1.center)}
        {(Dn-3.center) edge (B2.center)}
    ;
    \end{tikzpicture}
\end{equation}
where $n\geq 4$, can be contracted onto a poset of type \ref{egD4tilde}, namely a poset for which the underlying unoriented graph of the Hasse quiver admits the following shape:
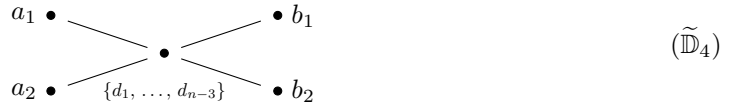
\begin{equation}\label{egD4tilde}\tag{$\qDtilde_4$}
    \begin{tikzpicture}[baseline=(current bounding  box.center)]
        \node[nodeDots, label=left:{$a_1$}] (A1) at (0,.5) {};
        \node[nodeDots, label=left:{$a_2$}] (A2) at (0,-.5) {};
        \node[nodeDots] (D) at (1.5,0){};
        \node at (D) [yshift=-.5cm,scale=.7]{$\{d_1,\,\ldots,\,d_{n-3}\}$};
        \node[nodeDots, label=right:{$b_1$}] (B1) at (3,.5) {};
        \node[nodeDots, label=right:{$b_2$}] (B2) at (3,-.5) {};
        
        \draw[shorten <=5pt, shorten >=5pt] 
            (A1) edge (D) 
            (A2) edge (D)
            (D) edge (B1)
            (D) edge (B2)
        ;
    \end{tikzpicture}
\end{equation}
\end{example}

We say that an incidence $\K$-algebra of a poset is \textit{minimally representation-infinite} if it is representation-infinite and all of its proper subposets and non-trivial contractions have representation-finite incidence $\K$-algebras. Now that the reduction techniques in \Cref{lem:Lou75} are established, it is clear that all representation-infinite incidence $\K$-algebras can be reduced to a minimal one, and that it suffices to classify the minimally representation-infinite incidence $\K$-algebras in order to distinguish representation-finiteness from representation-infiniteness. 
 \begin{theorem}[{\cite[Theorems 1.4 and 1.5]{Lou75}}]\label{thm:Lou75}
    A finite poset $P$ has a representation-finite incidence $\K$-algebra if and only if it cannot be reduced to a poset whose Hasse quiver is taken from the list in \Cref{tab:Loupias frames}, nor from the list of the opposites of these. In other words, any minimally representation-infinite incidence $\K$-algebra falls into one of these twelve types.
    
\captionsetup[table]{skip=10pt}
\begin{table}[ht!]
    \centering

    \begin{tabular}{cc|cc}
         $\qAtilde_3$& \QuiverAthreeDraw & $\qDtilde_4$  & \QuiverDfourDraw \\[3pt] \hline
         &&& \\[-6pt]
         $\qEtilde_6$& \QuiverEsixDraw & $\qEtilde_7$ & \QuiverEsevenDraw \\[3pt] \hline
         &&& \\[-6pt]
         $\qEtilde_8$& \QuiverEeightDraw & $\qR_1$ & \QuiverRoneDraw \\[3pt]  \hline
         &&& \\[-6pt]
         $\qR_2$& \QuiverRtwoDraw & $\qR_3$ & \quiverRthreeDraw \\[3pt] \hline
         &&& \\[-6pt]
         $\qR_4$& \QuiverRfourDraw & $\qR_5$ & \QuiverRfiveDraw \\[3pt] \hline
         &&& \\[-6pt]
         $\qR_6$ & \QuiverRsixDraw & $\qR_7$ &  \QuiverRsevenDraw \\[3pt] 
    \end{tabular}
    \caption{Hasse quivers of posets having minimally representation-infinite incidence $\K$-algebras. Undirected edges indicate that the arrows may go in either direction.}
    \label{tab:Loupias frames}
\end{table}
\end{theorem}

\begin{corollary}\label{cor:Lou75concealed}
    A finite poset $P$ has a representation-infinite incidence $\K$-algebra if and only if it can be reduced to a poset $L$ for which the incidence algebra $\incidencealgebra{\K}{L}$ is tame concealed. 
\end{corollary}
\begin{proof}
    Suppose that $P$ has a representation-infinite incidence $\K$-algebra. Then by \cref{thm:Lou75} $P$ can be reduced to one of the posets listed in \Cref{tab:Loupias frames}.
    It can be shown that the incidence $\K$-algebras of the listed posets are all tame concealed. Indeed, all of them (and their opposites) appear on Happel--Vossieck's list of tame concealed algebras \cite{HV83}, where we find that the incidence algebra of type $\qR_i$, where $1\leq i\leq 7$, is tame concealed of type $\qEtilde$. This proves necessity.
    
    Suppose that a finite number of subtractions and contractions of $P$ yields a poset $L$ such that $\incidencealgebra{\K}{L}$ is tame concealed. Since tame concealed algebras are representation-infinite, it follows from \Cref{lem:Lou75} that $P$ has a representation-infinite incidence $\K$-algebra. This proves the converse.
\end{proof}

\begin{example}\label{ex:repfiniteAnTensorAm}
    To familiarize the reader with the notions recalled above, we will now apply \Cref{thm:Lou75} to rediscover when the posets of the following form
    \begin{equation*}
        \prod\limits_{i=1}^\ell \totallyordered{n_i}
    \end{equation*}
    have representation-finite incidence $\K$-algebras, as originally shown by Leszczyński \cite[Proposition 2.1 and Theorem 2.4]{Les94}. We remind the reader that the poset $\totallyordered{n_i}$ is totally ordered and has $n_i$ elements.
    
    We consider posets of the form $\totallyordered{n}\times \totallyordered{m}$ first, where $n\leq m$. If $n\geq 3$, then $\totallyordered{n}\times \totallyordered{m}$ has a subposet of type $\qDtilde_4$, and consequently $\totallyordered{n}\times \totallyordered{m}$ is of infinite representation type. The poset $\totallyordered{2}\times \totallyordered{5}$; also has a representation-infinite incidence $\K$-algebra, since the subposet displayed by the solid nodes and arrows below
    
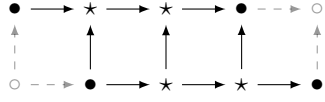
\begin{equation*}\begin{tikzpicture}
        \node (P-1-2) at (0,0) []{$\bullet$};
        \node (P-2-2) at (1,0) [] {$\star$};
        \node (P-3-2) at (2,0) [] {$\star$};
        \node (P-4-2) at (3,0) [] {$\bullet$};
        \node (P-5-2) at (4,0) [opacity=.4] {$\circ$};

        \node (P-1-1) at (0,-1) [opacity=.4] {$\circ$}; 
        \node (P-2-1) at (1,-1) [] {$\bullet$};
        \node (P-3-1) at (2,-1) [] {$\star$};
        \node (P-4-1) at (3,-1) [] {$\star$};
        \node (P-5-1) at (4,-1) [] {$\bullet$};
        
        \foreach \x [evaluate=\x as \xx using int(\x+1)] in {1,2,3}{
            \draw[-latex] (P-\x-2)--(P-\xx-2);
            \draw[-latex] (P-\xx-1)--(P-\xx-2);
        }
        \foreach \x [evaluate=\x as \xx using int(\x+1)] in {2,3,4}{
            \draw[-latex] (P-\x-1)--(P-\xx-1);
        }
        \draw[-latex,opacity=.4,dashed] (P-1-1)--(P-2-1);
        \draw[-latex,opacity=.4,dashed] (P-1-1)--(P-1-2);
        \draw[-latex,opacity=.4,dashed] (P-5-1)--(P-5-2);
        \draw[-latex,opacity=.4,dashed] (P-4-2)--(P-5-2);
    \end{tikzpicture}
    \end{equation*}
    can be contracted onto the poset
    
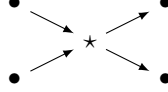
\begin{equation*}
        \begin{tikzpicture}
            \node (a) at (0,0) [] {$\star$};
            \node (b) at (-1,.5) [] {$\bullet$};
            \node (c) at (-1,-.5) [] {$\bullet$};
            \node (d) at (1,.5) [] {$\bullet$};
            \node (e) at (1,-.5) [] {$\bullet$};
            \draw[-latex] (b)--(a);
            \draw[-latex] (c)--(a);
            \draw[-latex] (a)--(d);
            \draw[-latex] (a)--(e);
        \end{tikzpicture}
    \end{equation*}
    which is of type $\qDtilde_4$. A poset of the form $\totallyordered{2}\times\totallyordered{m}$ can be reduced to $\totallyordered{2}\times\totallyordered{5}$ as long as $m\geq 5$, whence all of these posets have representation-infinite incidence $\K$-algebras. 
    One checks that $\totallyordered{2}\times \totallyordered{4}$ is iterated tilted of type $\qE_8$, so it has a representation-finite incidence $\K$-algebra. If $n=2$ and $m\leq 4$, one uses \Cref{thm:Lou75} to argue that $\totallyordered{n}\times \totallyordered{m}$ has a representation-finite incidence $\K$-algebra.
    The remaining cases are those where $n=1$, which yields the class of finite totally ordered posets, and all of these certainly have a representation-finite incidence $\K$-algebra. 

    The poset $\totallyordered{2}\times \totallyordered{2} \times \totallyordered{2}$ has a representation-infinite incidence $\K$-algebra; indeed, the subposet obtained by removing the maximal and minimal elements can be contracted onto a poset of type $\qAtilde_3$, as shown below. 
    \[
    \begin{tikzpicture}[scale=.9,every node/.style={scale=.9}]
        \node (A) at (0,0) [] {\begin{tikzpicture}
        \node (P-0-0-0) at (0,0) [opacity=.4] {$\circ$};
        \node (P-1-0-0) at (2,0) [] {$\bullet$};
        \node (P-0-1-0) at (1,.5) [] {$\bullet$};
        \node (P-0-0-1) at (0,2) [] {$\star$};
        
        \node (P-1-1-0) at (3,.5) [] {$\bullet$};
        \node (P-0-1-1) at (1,2.5) [] {$\star$};
        \node (P-1-0-1) at (2,2) [] {$\star$};
        \node (P-1-1-1) at (3,2.5) [opacity=.4]{$\circ$};

        \draw[-latex,opacity=.4,dashed] 
        {(P-0-0-0) edge (P-1-0-0)} 
        {(P-0-0-0) edge (P-0-1-0)} 
        {(P-0-0-0) edge (P-0-0-1)}
        ;
        \draw[-latex] 
        {(P-0-1-0) edge (P-1-1-0)}
        {(P-0-1-0) edge (P-0-1-1)}
        ;
        \draw[line width=3pt,white] 
        {(P-1-0-0) edge (P-1-1-0)}
        {(P-1-0-0) edge (P-1-0-1)}
        ;
        \draw[-latex] 
        {(P-1-0-0) edge (P-1-1-0)}
        {(P-1-0-0) edge (P-1-0-1)}
        ;
        \draw[line width=3pt,white] 
        {(P-0-0-1) edge (P-1-0-1)}
        {(P-0-0-1) edge (P-0-1-1)}
        ;
        \draw[-latex] 
        {(P-0-0-1) edge (P-1-0-1)}
        {(P-0-0-1) edge (P-0-1-1)}
        ;
        \draw[-latex,opacity=.4,dashed]
        {(P-1-1-0) edge (P-1-1-1)}
        {(P-1-0-1) edge (P-1-1-1)}
        {(P-0-1-1) edge (P-1-1-1)}
        ;
    \end{tikzpicture}};
    \node (B) at (6,0) [] {\begin{tikzpicture}
        \node (minimal1) at (1,1) [] {$\bullet$};
        \node (minimal2) at (1,-1) [] {$\bullet$};
        \node (maximal1) at (0,0) [] {$\star$};
        \node (maximal2) at (2,0) [] {$\bullet$};
        \draw[-latex]
        {(minimal1) edge (maximal1)}
        {(minimal1) edge (maximal2)}
        {(minimal2) edge (maximal1)}
        {(minimal2) edge (maximal2)}
        ;
    \end{tikzpicture}};
    \draw[decorate,decoration={zigzag,post=lineto,
    post length=2pt,pre length=2pt},->] (A)--(B);
    \end{tikzpicture}
    \]
    In fact, by \cite[Lemma 3.5]{CR19} (see also \cite[Example 18.6.2]{Len99}) we have that $\totallyordered{2}\times\totallyordered{2}\times\totallyordered{2}$ has a tame incidence $\K$-algebra.
    More generally, if we take any tensor product of three or more totally ordered sets, all with cardinality greater than or equal to 2, the result will always have a representation-infinite incidence $\K$-algebra. Indeed, they can all be reduced to $\totallyordered{2} \times \totallyordered{2}\times \totallyordered{2}$.
\end{example}

We now turn to the question of $\tau$-tilting finiteness of incidence algebras. 
\emph{A priori}, the $\tau$-tilting finiteness of an incidence algebra $\incidencealgebra{\K}{P}$ depends both on the field $\K$ and the poset $P$. However, when \Cref{thm:tautfiniteposets} has been proven, we will have established that it only depends on $P$; by \Cref{rem:repfin_allfields}, the representation-finiteness of incidence $\K$-algebras does not depend on the field.

The proof of \Cref{thm:tautfiniteposets} will rely on a variation of \Cref{lem:Lou75}. The statement in \Cref{lem:Pla19}\eqref{lem:Pla19.0} is slightly more general than the cited reference \cite[Corollary 2.2]{Pla19}, but the proof will remain identical, \emph{mutatis mutandis}. The key ingredient in the proof is the fact that a $\K$-algebra is $\tau$-tilting finite precisely when it has finitely many isomorphism classes of bricks \cite[Theorem~1.4]{DIJ19}.  Recall that an object $B$ in a $\K$-linear category $\mathcal{U}$ is a \textit{brick} if its endomorphism algebra is a division algebra over $\K$. 

\begin{lemma}\label{lem:Pla19}
    Let $\K$ be a field and let $\Lambda$ be a $\tau$-tilting finite $\K$-algebra. 
    \begin{enumerate}
    \setcounter{enumi}{-1}
        \item\label{lem:Pla19.0} Let $\cU\hookrightarrow \Modf(\Lambda)$ be a fully faithful functor of $\K$-linear categories. Then $\cU$ contains finitely many bricks up to isomorphism. In particular, if $\cU$ is equivalent to a module category $\Modf(\Lambda')$, then $\Lambda'$ is $\tau$-tilting finite.
    \end{enumerate}
    Let $P$ be a finite poset such that the incidence $\K$-algebra $\incidencealgebra{\K}{P}$ is $\tau$-tilting finite. 
    \begin{enumerate}
        \item\label{lem:PlaLou75.1} Any subposet of $P$ has a $\tau$-tilting finite incidence $\K$-algebra.
        \item\label{lem:PlaLou75.2} Any contraction of $P$ has a $\tau$-tilting finite incidence $\K$-algebra.
    \end{enumerate}
    \begin{proof}
        As noted above, the assertion in \eqref{lem:Pla19.0} is a slight generalization of work of Plamondon \cite[Corollary~2.2]{Pla19}, and is  proved in an identical manner.
        The other assertions then follow from the existence of the fully faithful functors in \eqref{eq:fullyfaithful_subtr} and \eqref{eq:fullyfaithful_contr}.
    \end{proof}
\end{lemma}

\begin{theorem}\label{thm:tautfiniteposets}
    Let $P$ be a finite poset. Then $\incidencealgebra{\K}{P}$ is $\tau$-tilting finite if and only if it is representation-finite.
\end{theorem}
\begin{proof}
    All representation-finite algebras are $\tau$-tilting finite. We prove that the converse assertion holds for incidence $\K$-algebras of finite posets.

    Suppose that $P$ has a representation-infinite incidence $\K$-algebra. By \Cref{cor:Lou75concealed}, one can reduce $P$ to a poset $L$ such that the incidence $\K$-algebra $\incidencealgebra{\K}{L}$ is tame concealed. Since tame concealed algebras are $\tau$-tilting infinite, as we remarked in \Cref{lem:concealed_tau_inf}, it follows directly from \Cref{lem:Pla19}\eqref{lem:Pla19.0} that $P$ has a $\tau$-tilting infinite incidence $\K$-algebra.
\end{proof}

There are other approaches to proving \Cref{thm:tautfiniteposets}. For instance, one can appeal to the fact that a simply connected $\K$-algebra is $\tau$-tilting finite if and only if it is representation-finite \cite[Theorem 1.1]{Wan22}. Since all minimally representation-infinite incidence $\K$-algebras, except $\qAtilde_3$, are simply connected, we immediately deduce the $\tau$-tilting infiniteness for eleven of the twelve cases listed in \Cref{tab:Loupias frames}. The exception, namely the poset $\qAtilde_3$, has a hereditary and representation-infinite incidence $\K$-algebra, and such $\K$-algebras are $\tau$-tilting infinite (cf. \cref{lem:concealed_tau_inf}), whereby we conclude the alternative proof.

The key insight leading to \Cref{thm:tautfiniteposets} was \Cref{cor:Lou75concealed}, where we showed that minimally representation-infinite incidence $\K$-algebras are concealed. Since any proper quotient of a tame concealed $\K$-algebra is representation-finite \cite[Proposition 6]{Bon17}, we deduce that these incidence $\K$-algebras are \emph{minimally $\tau$-tilting infinite $\K$-algebras}, which is to state that they are $\tau$-tilting infinite $\K$-algebras for which all proper quotients are $\tau$-tilting finite \cite{MP23}. 

\begin{example}\label{ex:tautfAnTensorAm}
Mitamoto and Wang have previously classified posets of the form $\totallyordered{n}\times \totallyordered{m}$ that have $\tau$-tilting finite incidence $\K$-algebras \cite{MW21}. This is achieved by appealing to the fact that path $\K$-algebras of the $\qA_n$-quivers are simply connected.
In \Cref{ex:repfiniteAnTensorAm} it was stated that $\totallyordered{n}\times \totallyordered{m}$ is representation-finite if and only if
\begin{equation}\label{eq:tautfAnTensorAm}
       1\in  \{n,m\} \quad \text{or} \quad \{n,m\} \in \big\{\{2,2\}, \{2,3\},\{2,4\} \big\}.
    \end{equation}
    By \Cref{thm:tautfiniteposets}, it follows that the condition stated in \eqref{eq:tautfAnTensorAm} is necessary and sufficient for $\totallyordered{n}\times \totallyordered{m}$ to have a $\tau$-tilting finite incidence $\K$-algebra.   
\end{example}

\begin{remark}
    A related concept to $\tau$-tilting finiteness is that of silting discreteness \cite[Definition~3.6]{Aihara13}. A silting discrete $\K$-algebra is $\tau$-tilting finite, but the converse is not true in general.
    There are incidence $\K$-algebras that are $\tau$-tilting finite, but not silting discrete. For example, consider the incidence $\K$-algebra of the poset with the Hasse quiver displayed below.
    
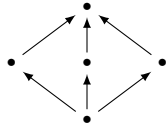
\[
    \begin{tikzpicture}
        \node (A) at (2.5,0) [scale=.7] {$\bullet$};
        \node (B) at (1.5,.75) [scale=.7]{$\bullet$};
        \node (C) at (1.5,0) [scale=.7]{$\bullet$};
        \node (D) at (1.5,-.75) [scale=.7]{$\bullet$};
        \node (E) at (.5,0) [scale=.7]{$\bullet$};
        \draw[-latex] (D)--(E);
        \draw[-latex] (D)--(C);
        \draw[-latex] (D)--(A);
        \draw[-latex] (E)--(B);
        \draw[-latex] (C)--(B);
        \draw[-latex] (A)--(B);
    \end{tikzpicture}
    \]
    \revised{One observes that the incidence $\K$-algebra has a representation-finite incidence $\K$-algebra, which is hence also $\tau$-tilting finite. It is, however, tilted of type $\qDtilde_4$. Since silting discreteness only depends on the derived category, and a path $\K$-algebra of a quiver of type $\qDtilde_4$ has a $\tau$-tilting infinite incidence $\K$-algebra, we may conclude that the incidence $\K$-algebra is not silting discrete.}
\end{remark}

\section{\texorpdfstring{$\mathbf{g}$}{g}-tameness of incidence algebras and concealed algebras}\label{sec:g}
For the remainder of this text, we assume the field $\K$ to be algebraically closed. The main result in this section will be the following. 
\begin{reptheorem}{thm:notgtame}
    Let $B$ be a wild concealed $\K$-algebra. Then $B$ is not \g-tame.
\end{reptheorem}
We then deduce \Cref{cor:concealed_gvecfan}, asserting that the incidence $\K$-algebra of a finite simply connected poset is $\mathbf{g}$-tame if and only if it is tame.
We first recall the notion of $\mathbf{g}$-tameness, so that we may develop the proof of \Cref{thm:notgtame} in Sections \ref{sec:3.1-concealed_typeA} and
\ref{sec:3.2-wild_not_gtame} below. 

\subsection{Preliminaries on $\g$-vector fans}

Before recalling the definition of the $\mathbf{g}$-vector fan of a finite-dimensional $\K$-algebra, we give a brief exposition on Grothendieck groups, as well as the combinatorial notion of polyhedral fans \cite{Zie95}.

\introheaderTwo{Grothendieck groups.} Let $(\mathcal{C},\mathcal{E})$ denote an essentially small Quillen exact category, let $\mathrm{iso}(\mathcal{C})$ denote the set of isomorphism classes of objects in $\mathcal{C}$, \revised{and let $[A]$ denote the isomorphism class of an object $A\in \mathcal C$.} The \textit{Grothendieck group} of $(\mathcal{C},\mathcal{E})$ is defined as follows:
\begin{equation*}
\mathrm{K}_0(\mathcal{C},\mathcal{E}) \defeq \left. \Z^{\oplus \mathrm{iso}(\mathcal{C})} \middle/ \big\langle[B] - [A] - [C] \sth
	 \begin{tikzcd}
			A \arrow[r,tail] & B \arrow[r,two heads] & C
	\end{tikzcd}
	\textit{ is an $\mathcal{E}$-exact sequence} \big\rangle.\right.
\end{equation*}
\revised{By abuse of notation we will write $[A]$ to refer to the image of $[A]$ in $K_0(\mathcal C, \mathcal E)$.}

For a finite-dimensional $\K$-algebra $\Lambda$, we consider two examples.  The first is $\Modf(\Lambda)$, the abelian category of finite-dimensional $\Lambda$-modules. The second is $\proj(\Lambda)$, the full subcategory of $\Modf(\Lambda)$ consisting of finite-dimensional projective $\Lambda$-modules, which is equipped with the split-exact structure. One sees that $\mathrm{K}_0(\proj\Lambda)$ is a free abelian group of rank $n$, where $n$ is the number of isomorphism classes of indecomposable projective $\Lambda$-modules. 
Using the indecomposable finite-dimensional projective $\Lambda$-modules as a basis for $\mathrm{K}_0(\proj\Lambda)$, we can identify the group $\mathrm{K}_0(\proj\Lambda)$ with $\mathbb{Z}^n$.

There is an injection
\begin{equation}\label{eq:gvectormasp}
    \begin{tikzcd}
        \g^{(-)}\colon \taurigidpair{\Lambda}\arrow[r,hook] & \mathrm{K}_0(\proj\Lambda)
    \end{tikzcd}
\end{equation}
from the set of $\tau$-rigid pairs in $\Modf(\Lambda)$ into the Grothendieck group of $\proj(\Lambda)$ \cite[Theorem 5.5]{AIR14}.
This injection can be explicitly defined as follows: let $(M,Q)$ be a $\tau$-rigid pair and let
\(
P_1 \longrightarrow P_0
\)
be a minimal projective presentation of $M$. Set $\g^{(M,Q)}\defeq [P_0]-[P_{1}]-[Q]\in \mathrm{K}_0(\proj\Lambda)$. 
We define the \textit{real Grothendieck group} of $\Lambda$ by 
$$
    \mathrm{K}_0(\proj\Lambda)_{\mathbb{R}} \defeq \mathrm{K}_0(\proj\Lambda)\otimes_{\mathbb{Z}}\mathbb{R}.
$$ 
As an $\R$-vector space, it is isomorphic to $\R^n$, where $n$ is the rank of the algebra.
For a $\tau$-rigid pair $(M,Q)$, we call the element $\g^{(M,Q)}\otimes 1\in \mathrm{K}_0(\proj\Lambda)_{\R}$ its \textit{$\g$-vector}. Since the map in \eqref{eq:gvectormasp} injects, we have that $\g$-vectors determine $\tau$-rigid pairs uniquely.
\revised{An indecomposable} $\tau$-rigid $\Lambda$-module can thus be identified with the ray of \revised{its \g-vector} in $\mathrm{K}_0(\proj\Lambda)_{\mathbb{R}}$. Equivalently, we can regard it as a point on the unit $(n-1)$-sphere $\S^{n-1}$. 

\introheaderTwo{Polyhedral fans.} Fixing a positive integer $n$, consider an $n$-dimensional vector space $V$.
Given a set of vectors $S=\{v_1 , \dots , v_{\ell}\}\subseteq V$, its positive span $C_S\subseteq \R^n$ is called a \textit{polyhedral cone} in $V$, or simply a \textit{cone}. Such a cone is said to be \textit{simplicial} if it is of the form $C_S$, where $S$ is a set of linearly independent vectors in $V$. The \textit{dimension} of a simplicial cone is the dimension of the subspace spanned by it. 
The faces of a cone can be defined through intersecting the cone with certain half-spaces \cite[Defintion~2.1]{Zie95}. When $C_S$ is a simplicial cone, the faces are given by $C_T$ for subsets $T \subseteq S$.
A \textit{polyhedral fan} in $V$ is given by a family of polyhedral cones in $V$ denoted $\mathcal{S} = \{{C}_X\}_X$, enjoying the following two properties:
\begin{enumerate}
    \item The family $\mathcal{S}$ is closed under the formation of faces.
    \item The intersection of two elements $C_1$ and $C_2$ in $\mathcal{S}$ is in $\mathcal{S}$, and it is a face to both $C_1$ and $C_2$.
\end{enumerate} 
A polyhedral fan $\mathcal{S}$ is \textit{simplicial} if all of its cones are simplicial. 
We say that $\mathcal{S}$ is \textit{finite} if the number of elements in $\mathcal{S}$ is finite, and \textit{complete} if every vector in $V$ lies inside some cone $C\in \mathcal{S}$.

\introheaderTwo{$\g$-vector fans and $\g$-tameness.} For a (possibly decomposable) $\tau$-rigid pair $(M,Q)$ in $\Modf(\Lambda)$, let $C_{(M,Q)}$ denote the cone in $\mathrm{K}_0(\proj\Lambda)_{\mathbb{R}}$ positively spanned by the \g-vectors of the indecomposable direct summands of $(M,Q)$. 
The \textit{$\g$-vector fan} of $\Lambda$ is defined as a the polyhedral fan consisting of the cones $C_{(M,Q)}$, where $(M,Q)$ is a $\tau$-rigid pair $(M,Q)$ in $\Modf(\Lambda)$. We denote the $\g$-vector fan of $\Lambda$ by $\gvectorfan{\Lambda}$. It is a simplicial and essential fan. 

\begin{theorem}[{\cite[Theorem~4.7]{Asa21}}]\label{thm:CompleteGfan}
Let $\Lambda$ be a finite-dimensional $\K$-algebra. 
The $\g$-vector fan $\gvectorfan{\Lambda}$ is finite precisely when it is complete. Equivalently, the $\K$-algebra $\Lambda$ is $\tau$-tilting finite if and only if the \g-vector fan $\gvectorfan{\Lambda}$ is complete.
\end{theorem} 

Having classified all $\tau$-tilting finite incidence $\K$-algebras of finite posets in \Cref{thm:tautfiniteposets}, we deduce from \Cref{thm:CompleteGfan} that we can characterize the incidence $\K$-algebras of finite posets admitting a complete \g-vector fan. We now move on to the case where the \g-vector fan may be incomplete. 

\begin{definition}\cite{AY23}
    We say that a finite-dimensional $\K$-algebra $\Lambda$ is \textit{$\g$-tame} if its $\g$-vector fan is dense in $\R^{n}$, i.e. for any $x\in \R^n$ and for all $\varepsilon>0$, the $\varepsilon$-ball centered at $x$ intersects a cone in $\gvectorfan{\Lambda}$.
\end{definition}

By \Cref{thm:CompleteGfan}, all $\tau$-tilting finite $\K$-algebras are \g-tame. Furthermore, any tame $\K$-algebra is $\mathbf{g}$-tame \cite{PY23}. However, the converses of these two assertions are false; there are examples of both tame $\tau$-tilting finite $\K$-algebras (for example a tame local $\K$-algebra) and wild $\K$-algebras that are \g-tame (for example a wild local $\K$-algebra). For some nice classes of $\K$-algebras, the notions of tameness and \g-tameness coincide.

\begin{proposition}[{\cite[Theorem 5.1(4)]{Hil06}, \cite[Theorem 1.2]{Yur23}}]\label{prop:g_tam}
    Let $Q$ be a finite quiver. Then the path algebra $\K Q$ is $\g$-tame if and only if it is tame.
\end{proposition}
Our contribution to the theory of $\g$-tameness in \Cref{thm:notgtame} will build on known results about hereditary $\K$-algebras. Indeed, we will generalize \Cref{prop:g_tam} to the context of concealed $\K$-algebras.

\subsection{Reduction from wild to hyperbolic type}
In order to prove \Cref{thm:notgtame}, we devise a reduction technique from concealed $\K$-algebras to smaller concealed $\K$-algebras. Most of the current section is due to Happel--Unger \cite{HU89}, but there are some edge cases that we need to handle more carefully. Specifically, concealed $\K$-algebras of type $\qAtildetilde$ will need special attention.

The following lemma will be paramount. Recall that a subposet $P'$ of a poset $P$ is \textit{convex in $P$} if the following condition holds: if $x,y\in P'$ then $[x,y]\subseteq P'$, i.e. $z\in P'$ provided that $x\leq z\leq y$. Note that the assertion in \eqref{lem.PY23.3.11} below can be strengthened to include general ideal quotients \cite[Corollary 7.8]{AY23}.

\begin{lemma}\label{lem:convidem}
    \begin{enumerate}
        \item[] 
        \item\label{lem.PY23.3.11} Let $\Lambda$ be a $\g$-tame $\K$-algebra and let $e\in\Lambda$ be an idempotent. Then the factor algebra $\Lambda/(e)$ is $\g$-tame \cite[Proposition 3.11]{PY23}.
        \item\label{lem:convidem2} Let $P$ be a finite poset, and let $L$ be a convex subposet of $P$. 
    If $P$ has a $\g$-tame incidence $\K$-algebra, so does $L$. 
    \end{enumerate}
\end{lemma}
\begin{proof}
    We deduce assertion \eqref{lem:convidem2} from assertion \eqref{lem.PY23.3.11}.
    Let $e = \sum_{x\in L} e_x$ be the sum of the primitive idempotents corresponding to the elements in $L$. Since $L$ is convex, it follows that $\incidencealgebra{\K}{L}$ is isomorphic to $\incidencealgebra{\K}{P}/(1-e)$. Since $1-e$ is an idempotent in the $\g$-tame algebra $\incidencealgebra{\K}{P}$, we have by \eqref{lem.PY23.3.11} that the algebra $\incidencealgebra{\K}{P}/(1-e)$ is $\g$-tame. We conclude that $\incidencealgebra{\K}{L}$ is $\g$-tame.
\end{proof}

\begin{definition}\label{definition:hyperbolic}
    A finite-dimensional $\K$-algebra $\Lambda$ is called \textit{hyperbolic} if the following conditions hold.
    \begin{enumerate}
        \item The $\K$-algebra $\Lambda$ is wild.
        \item For all primitive idempotents $e\in\Lambda$, the factor algebra $\Lambda/(e)$ is tame.
    \end{enumerate}
    An acyclic quiver $Q$ is \textit{hyperbolic} if the path algebra $\K Q$ is hyperbolic, or equivalently if $Q$ is wild and any proper full subquiver of $Q$ is tame. 
    If $Q$ is a hyperbolic quiver, a concealed $\K$-algebra of type $Q$ is called \emph{hyperbolically concealed of type $Q$}.\footnote{Our terminology is that of Kac \cite{Kac83} rather than that of Happel--Unger \cite{HU89} and Kerner \cite{Ker88}, who refer to hyperbolic quivers as \textit{minimally wild}.}
\end{definition}

\begin{remark}\label{rem:Kerner}
     It does not follow directly from the definition that hyperbolically concealed $\K$-algebras are hyperbolic, but this is the case \cite[Proposition 4.1]{Ker88}. Kerner provides an exhaustive list of hyperbolic quivers \cite[Remark following Theorem 4.1]{Ker88}, which enables a classification of hyperbolically concealed $\K$-algebras by listing all concealed $\K$-algebras that arise from these quivers. Such a classification is given by Unger \cite{unger_concealed_1990}.
\end{remark}

In this section we prove the following result.
\begin{theorem}\label{thm:red_to_hyperbolic}
    Let $Q$ be a wild quiver, and let $B$ be a concealed $\K$-algebra of type $Q$. Then there exists an idempotent $e\in B$ such that $B/(e)$ is concealed of hyperbolic type.
    \end{theorem}

\begin{definition}
    Let $Q$ be a quiver.
    \begin{enumerate}
        \item A vertex $v$ in a quiver $Q$ is called a \textit{tip} if it has precisely one neighboring vertex, to which it is connected by exactly one edge.\footnote{Clearly, a tip must either be a sink or a source of the quiver.} 
        \item The quiver $Q$ is \textit{twice edge-connected} if every vertex of $Q$ is adjacent to at least two edges, or equivalently if $Q$ does not have a tip.
    \end{enumerate}
\end{definition}

The following results of Happel--Unger will be instrumental in our proof of \Cref{thm:red_to_hyperbolic}.
\begin{lemma}\cite{HU89}\label{lem:HU89}
    Let $Q$ be a connected acyclic wild quiver with more than two vertices, let $T$ be a postprojective tilting $\K Q$-module and let $B=\End_{\K Q}(T)\op$ be the corresponding concealed $\K$-algebra of type $Q$.
    \begin{enumerate}
        \item\label{lem:HU89_1.1} If $Q$ is not twice edge-connected and $v$ is a tip of $Q$, then the subquiver $Q' = Q \setminus \{v\}$ is connected and representation-infinite \cite[§1.1]{HU89}.
        \item\label{lem:HU89_1.2} If $Q$ is twice edge-connected, there exists a vertex $v$ in $Q$ such that the subquiver $Q'= Q\setminus \{v\}$ is connected and representation-infinite \cite[§1.2]{HU89}.
        \item\label{lem:HU89_1.1+2} Let $v$ be as in \eqref{lem:HU89_1.1} or \eqref{lem:HU89_1.2}, depending on whether $Q$ is twice edge-connected. There exists an indecomposable direct summand $T_1$ of $T$ belonging to the $\tau$-orbit of the projective $\K Q$-module $P(v)$ \cite[§1.1 and §1.2]{HU89}.
    \end{enumerate}
\end{lemma}

\begin{lemma}\cite{HU89}\label{lem:remove_vertex_in_tau_orbit}
    Let $Q$ be a quiver, let $v$ be a vertex such that $Q' \defeq Q\setminus\{v\}$ is connected representation-infinite, and let $P(v)$ denote the indecomposable projective $\K Q$-module corresponding to $v$. Let $T$ be a postprojective tilting ${\K Q}$-module and let $T_1$ be an indecomposable direct summand of $T$ such that $T_1 = \tau^{-r}P(v)$ for some $r$. Let $B= \End_{\K Q}(T)\op$ and let $e$ be the idempotent such that $Be = \Hom(T, T_1)$. Then $B/(e)$ is concealed of type $Q'$.
    \begin{proof}
        This is proven in \cite{HU89}. Note that they assume that $v$ is a tip whenever $Q$ is not twice edge-connected, but that this assumption is only used to guarantee the existence of a a summand equal to $\tau^{-r}P(v)$, and is not necessary for the remaining argument. 
    \end{proof}
\end{lemma}

We will also need some technical lemmata.

\begin{lemma}\label{lem:hyperbolic_tips}
    Let $Q$ be a wild tree. Suppose that each subquiver of the form $Q\setminus\{v\}$, where $v$ is a tip, is a tame quiver. Then $Q$ is a hyperbolic quiver. 
\end{lemma}
\begin{proof}
    If $Q$ were not hyperbolic, there would exist a vertex $x$ that is not a tip such that $Q\setminus\{x\}$ is wild. The quiver $Q\setminus\{x\}$ has more than one connected component, and at least one of which is wild. Let $C_w$ be a wild connected component of $Q\setminus\{x\}$, let $C$ be a different connected component and let $v$ be a tip of $Q$ contained in $C$. Then $Q\setminus\{v\}$ contains $C_w$ as a subquiver, making $Q\setminus\{v\}$ wild, contradicting our assumption on the tips of $Q$. Having reached a contradiction under the assumption of non-hyperbolicity, we conclude that $Q$ must be hyperbolic.
\end{proof}

The previous two lemmata are already enough for us to reduce concealed algebras of tree-type. We now turn our attention to understanding postprojective modules in algebras of type $\qAtildetilde$. These observations seem to be a part of the folklore, but we find it beneficial to state and prove them explicitly.

\begin{lemma}\label{lem:at_most_two_summands_in_orbit_of_P(a)}
    Let $Q$ be a quiver of type $ \qAtildetilde_n$ for $n>3$, and $a\in Q_0$ the unique tip of the quiver. Consider an indecomposable $\K Q$-module $M$  in the $\tau$-orbit of $P(a)$, then for $r\geq 1$
    \(
    \Ext_{\K Q}^{1}({\tau^{-r}M},\,{M})=0 
    \) only if $r=2$. 
\end{lemma}
\begin{proof}
    Through the Auslander--Reiten-formula and $\tau^-$-shift we see that
    \[
    \Ext_{\K Q}^{1}({\tau^{-r}M},\,{M})\cong D\Hom_{\K Q}({P(a)},\,{\tau^{-r+1}P(a)}),
    \]
    so we may assume that $M=P(a)$, and we are concerned with when $\Hom_{\K Q}({P(a)},\,{\tau^{-r+1}P(a)})=0$. By iterated APR-tilting (see \cite{APR79}), we know that $\K Q$ is concealed of type $\qAtildetilde_{p,q}$ with $p\leq q$ such that the clockwise and anti-clockwise arrows are collected. For the purposes of considering Hom-spaces in the postprojective component, we may therefore assume that $Q$ is of the form indicated in \eqref{eq:Apq_spec}. Let us consider three cases: $p>2$, $p=2$, and $p=1$. In the first case, we have $Q$ of the form
    \begin{equation}\label{eq:Apq_spec}
    \begin{tikzpicture}[scale=.7, baseline=(current bounding  box.center)]
        \node[label=right:{$a$}] (a) at (3,0) {$\bullet$};
        \node[label=left:{$b$}] (b) at (0:1.5) {$\bullet$};
        \node[] (c) at (60:1.5) {$\bullet$};
        \node at (60:2) []{$\alpha$};
        \node[] (f) at (120:1.5) {$\bullet$};
        \node at (120:2) [] {$\beta$};
        
        \node (dotsStart) at (170:1.5) [] {};
        \draw[fill] (175:1.5) circle (.5pt);
        \draw[fill] (180:1.5) circle (.5pt);
        \draw[fill] (185:1.5) circle (.5pt);
        \node (dotsEnd) at (190:1.5) [] {};

        \node[] (e) at (240:1.5) {$\bullet$};
        \node at (240:2) [] {$\hat{\beta}$};
        \node[] (d) at (300:1.5) {$\bullet$};
        \node at (300:2) [] {$\hat{\alpha}$};
        
        \draw[-latex] (dotsStart) to[out=70,in=220] (f);
        \draw[-latex] (dotsEnd) to[out=-70,in=140] (e);
        \draw[-latex] (f) to[out=25,in=155] (c);
        \draw[-latex] (e) to[out=-25,in=-155] (d);
        \draw[-latex] (d) to[out=40, in=-100] (b);
        \draw[-latex] (c) to[out=-40, in=100] (b);
        \draw[-latex] (b)--(a);
    \end{tikzpicture}
        \end{equation}
   At the beginning of the postprojective component in the AR-quiver of $\K Q$, we obtain the below subquiver, where we denote projectives by their corresponding nodes in $Q$.
    \[
    \begin{tikzpicture}[yscale=.6,xscale=.6,every node/.style={scale=.8,minimum width={width"{$\tau^{-3}a$}"}}]
        \node (Pa) at (0,1.8) [] {$a$};
        \node (Pb) at (2,1) [] {$b$};
        \node (Pc) at (4,3) [] {$\alpha$};
        \node (Pf) at (6,5) [] {$\beta$};
        \node (Pd) at (4,-1) [] {${\hat{\alpha}}$};
        \node (Pe) at (6,-3) [] {${\hat{\beta}}$};
        
        \node (t-Pa) at (4,1.8) [] {$\tau^-a$};
        \node (t-Pb) at (6,1) [] {$\tau^-b$};
        \node (t-Pc) at (8,3) [] {$\tau^-\alpha$};
        \node (t-Pd) at (8,-1) [] {$\tau^-{\hat{\alpha}}$};
        
        \node (t-2Pa) at (8,1.8) [] {$\tau^{-2}a$};
        \node (t-2Pb) at (10,1) [] {$\tau^{-2}b$};
        \node (t-3Pa) at (12,1.8) [] {$\tau^{-3}a$};

        \draw[->] (Pa)--(Pb);
        \draw[->] (Pb)--(t-Pa);
        \draw[->] (t-Pa)--(t-Pb);
        \draw[->] (t-Pb)--(t-2Pa);
        \draw[->] (t-2Pa)--(t-2Pb);
        \draw[->] (t-2Pb)--(t-3Pa);
        \draw[->] (Pb)--(Pc);
        \draw[->] (Pc)--(t-Pb);
        \draw[->] (t-Pb)--(t-Pc);
        \draw[->] (t-Pc)--(t-2Pb);
        \draw[->] (Pc)--(Pf);
        \draw[->] (Pf)--(t-Pc);
        \draw[->] (Pb)--(Pd);
        \draw[->] (Pd)--(t-Pb);
        \draw[->] (t-Pb)--(t-Pd);
        \draw[->] (t-Pd)--(t-2Pb);
        \draw[->] (Pd)--(Pe);
        \draw[->] (Pe)--(t-Pd);
        \draw[dashed] (Pa)--(t-Pa) (t-Pa)--(t-2Pa) (t-2Pa)--(t-3Pa) (Pb)--(t-Pb) (t-Pb)--(t-2Pb) (Pc)--(t-Pc) (Pd)--(t-Pd);
    \end{tikzpicture}
    \]
    Applying $\dim\Hom_{\K Q}({P(a)},\,{-})$ on the projectives and then using the exactness of almost split sequences, we can manually compute $\dim\Hom_{\K Q}({P(a)},\,{-})$ for the nodes:
    \[
    \begin{tikzpicture}[yscale=.4,xscale=.4,every node/.style={scale=.7}]
        \node (Pa) at (0,1.8) [] {$1$};
        \node (Pb) at (2,1) [] {$1$};
        \node (Pc) at (4,3) [] {$1$};
        \node (Pf) at (6,5) [] {$1$};
        \node (Pd) at (4,-1) [] {$1$};
        \node (Pe) at (6,-3) [] {$1$};
        
        \node (t-Pa) at (4,1.8) [] {$0$};
        \node (t-Pb) at (6,1) [] {$1$};
        \node (t-Pc) at (8,3) [] {$1$};
        \node (t-Pd) at (8,-1) [] {$1$};
        
        \node (t-2Pa) at (8,1.8) [] {$1$};
        \node (t-2Pb) at (10,1) [] {$2$};
        \node (t-3Pa) at (12,1.8) [] {$1$};

        \draw[->] (Pa)--(Pb);
        \draw[->] (Pb)--(t-Pa);
        \draw[->] (t-Pa)--(t-Pb);
        \draw[->] (t-Pb)--(t-2Pa);
        \draw[->] (t-2Pa)--(t-2Pb);
        \draw[->] (t-2Pb)--(t-3Pa);
        \draw[->] (Pb)--(Pc);
        \draw[->] (Pc)--(t-Pb);
        \draw[->] (t-Pb)--(t-Pc);
        \draw[->] (t-Pc)--(t-2Pb);
        \draw[->] (Pc)--(Pf);
        \draw[->] (Pf)--(t-Pc);
        \draw[->] (Pb)--(Pd);
        \draw[->] (Pd)--(t-Pb);
        \draw[->] (t-Pb)--(t-Pd);
        \draw[->] (t-Pd)--(t-2Pb);
        \draw[->] (Pd)--(Pe);
        \draw[->] (Pe)--(t-Pd);
        \draw[dashed] (Pa)--(t-Pa) (t-Pa)--(t-2Pa) (t-2Pa)--(t-3Pa) (Pb)--(t-Pb) (t-Pb)--(t-2Pb) (Pc)--(t-Pc) (Pd)--(t-Pd);
    \end{tikzpicture}
    \]
which shows us that $\Hom_{\K Q}({P(a)},\,{\tau^{-r}P(a)})$ is nonzero for $r=2,3$. Similar computation for $p=2$ and $p=1$ give the same conclusion:
\[
\begin{tikzpicture}
    \node at (0,0) [anchor=east]{\begin{tikzpicture}[yscale=.4,xscale=.4,every node/.style={scale=.7}]
        \node at (-1,1) [scale=1.5] {$p=2\colon$};
        \node (Pa) at (0,1.8) [] {$1$};
        \node (Pb) at (2,1) [] {$1$};
        \node (Pc) at (4,3) [] {$1$};
        \node (Pf) at (6,5) [] {$2$};
        \node (Pd) at (4,-1) [] {$1$};
        \node (Pe) at (6,-3) [] {$1$};
        
        \node (t-Pa) at (4,1.8) [] {$0$};
        \node (t-Pb) at (6,1) [] {$1$};
        \node (t-Pc) at (8,3) [] {$2$};
        \node (t-Pd) at (8,-1) [] {$1$};
        
        \node (t-2Pa) at (8,1.8) [] {$1$};
        \node (t-2Pb) at (10,1) [] {$3$};
        \node (t-3Pa) at (12,1.8) [] {$2$};

        \draw[->] (Pa)--(Pb);
        \draw[->] (Pb)--(t-Pa);
        \draw[->] (t-Pa)--(t-Pb);
        \draw[->] (t-Pb)--(t-2Pa);
        \draw[->] (t-2Pa)--(t-2Pb);
        \draw[->] (t-2Pb)--(t-3Pa);
        \draw[->] (Pb)--(Pc);
        \draw[->] (Pc)--(t-Pb);
        \draw[->] (t-Pb)--(t-Pc);
        \draw[->] (t-Pc)--(t-2Pb);
        \draw[->] (Pc)--(Pf);
        \draw[->] (Pf)--(t-Pc);
        \draw[->] (Pb)--(Pd);
        \draw[->] (Pd)--(t-Pb);
        \draw[->] (t-Pb)--(t-Pd);
        \draw[->] (t-Pd)--(t-2Pb);
        \draw[->] (Pd)--(Pe);
        \draw[->] (Pe)--(t-Pd);
        \draw[dashed] (Pa)--(t-Pa) (t-Pa)--(t-2Pa) (t-2Pa)--(t-3Pa) (Pb)--(t-Pb) (t-Pb)--(t-2Pb) (Pc)--(t-Pc) (Pd)--(t-Pd);
    \end{tikzpicture}};

    \node at (1,0) [] {and};
    
    \node at (1.5,0) [anchor=west] {\begin{tikzpicture}[yscale=.4,xscale=.4,every node/.style={scale=.7}]
        \node at (-3,1) [scale=1.5] {$p=1\colon$};
        \node (Pa) at (0,1.8) [] {$1$};
        \node (Pb) at (2,1) [] {$1$};
        \node (Pc) at (4,3) [] {$2$};
        \node (Pf) at (6,5) [] {$3$};
        \node (Pd) at (4,-1) [] {$1$};
        \node (Pe) at (6,-3) [] {$1$};
        
        \node (t-Pa) at (4,1.8) [] {$0$};
        \node (t-Pb) at (6,1) [] {$2$};
        \node (t-Pc) at (8,3) [] {$3$};
        \node (t-Pd) at (8,-1) [] {$2$};
        
        \node (t-2Pa) at (8,1.8) [] {$2$};
        \node (t-2Pb) at (10,1) [] {$5$};
        \node (t-3Pa) at (12,1.8) [] {$3$};

        \draw[->] (Pa)--(Pb);
        \draw[->] (Pb)--(t-Pa);
        \draw[->] (t-Pa)--(t-Pb);
        \draw[->] (t-Pb)--(t-2Pa);
        \draw[->] (t-2Pa)--(t-2Pb);
        \draw[->] (t-2Pb)--(t-3Pa);
        \draw[->] (Pb)--(Pc);
        \draw[->] (Pc)--(t-Pb);
        \draw[->] (t-Pb)--(t-Pc);
        \draw[->] (t-Pc)--(t-2Pb);
        \draw[->] (Pc)--(Pf);
        \draw[->] (Pf)--(t-Pc);
        \draw[->] (Pb)--(Pd);
        \draw[->] (Pd)--(t-Pb);
        \draw[->] (t-Pb)--(t-Pd);
        \draw[->] (t-Pd)--(t-2Pb);
        \draw[->] (Pd)--(Pe);
        \draw[->] (Pe)--(t-Pd);
        \draw[dashed] (Pa)--(t-Pa) (t-Pa)--(t-2Pa) (t-2Pa)--(t-3Pa) (Pb)--(t-Pb) (t-Pb)--(t-2Pb) (Pc)--(t-Pc) (Pd)--(t-Pd);
    \end{tikzpicture}};
\end{tikzpicture}
    \]
By the dual of \cite[Lemma 1]{Ung86}, this means there exists $\tau^-$-stable monomorphisms from $P(a)$ to $\tau^{-2}P(a)$ and $\tau^{-3}P(a)$. Now, if $r>1$ is an integer, then it can be written as $2x + 3y$ for $x,y\geq0$. Then the composite map
\[\begin{tikzcd}[column sep=0.5cm]
	{P(a)} & {\tau^{-2}P(a)} & \cdots & {\tau^{-2x}P(a)} & {\tau^{-(2x+3)}P(a)} & \cdots & {\tau^{-r}P(a)}
	\arrow[from=1-1, to=1-2]
	\arrow[from=1-2, to=1-3]
	\arrow[from=1-3, to=1-4]
	\arrow[from=1-4, to=1-5]
	\arrow[from=1-5, to=1-6]
	\arrow[from=1-6, to=1-7]
\end{tikzcd}\]
is a composition of monomorphisms, hence nonzero. Thus, the only value of $r\geq 1$ such that $\Hom_{\K Q}({P(a)},\,{\tau^{-r+1}P(a)})=0$ is $r=2$.
\end{proof}

\begin{lemma}\label{lem:postprojective_tilting_modules_of_type_Atildetilde}
    Let $Q$ be a quiver of type $\qAtildetilde_n$ for $n>3$, and $a\in Q_0$ the unique tip of the quiver with unique neighbour $b\in Q_0$. Let $M$ be any indecomposable $\K Q$-module in the $\tau$-orbit of $P(a)$, and $N$ any indecomposable $\K Q$-module in the $\tau$-orbit of $P(b)$, then $N\oplus M\oplus \tau^{-2}M$ is not rigid.

    Moreover, for any postprojective tilting module $T\in \mod{\K Q}$ and vertex $c\neq b$ of $\qAtildetilde_n$, there is a summand $T'$ of $T$ in the $\tau$-orbit of $P(c)$.
    \begin{proof}
    Like before, we may assume that $P(a)$ is a simple $\K Q$-module. By shifting with $\tau^-$ if necessary, we may assume that $M$ and $N$ are not projective. Then we have an almost split sequence
    \[\begin{tikzcd}
        0 \ar[r] & \tau M \ar[r] & X \ar[r] & M \ar[r] & 0
    \end{tikzcd}\]
    and an integer $r$ such that $\tau^r X = N$. If $r\geq 2$, then by \cref{lem:at_most_two_summands_in_orbit_of_P(a)} we have a morphism $\tau^{r+1}M \to \tau M$, and since every morphism between postprojective $\K Q$-modules is a composition of irreducible maps, this must factor through $\tau^{r+1}M \to \tau^r X$. Hence, there is a nonzero map $N \to \tau M$, meaning $\Ext_{\K Q}^1({M,\, N})\neq 0$. A similar argument applies for $r \leq -3$. We have a map from $M$ to $\tau^{r+1}M$ which must factor through $\tau^{r+1}X \to \tau^{r+1}M$, hence there is a nonzero map $M \to \tau N$. If $r$ is $1$ or $-2$ we can immediately see that we have non-zero maps $N \to \tau M$ and $M \to \tau N$ respectively. Leaving the only possibilities as $r=0$ or $r=-1$. Doing similar analysis replacing $M$ with $\tau^{-2}M$ leaves the only possibilities to be $r=-2$ or $r-3$. Since these restrictions do not overlap, there can be no such $r$. Hence $N \oplus M \oplus \tau^{-2} M$ is not rigid.

    To prove the second part, we first wish to establish that $X \oplus \tau^{-r}X$ is never rigid when $X$ is a postprojective indecomposable $\K Q$-module which is outside of the $\tau$-orbit of $P(a)$ and $r\geq 1$.
    Let $X$ and $Y$ be postprojective indecomposable $\K Q$-modules outside of the $\tau$-orbit of $P(a)$, and let $f\colon X\to Y$ be an irreducible morphism. We wish to show that $f$ is injective. If $Y$ is projective, then this is immediate, so let $Y$ be non-projective.
    Then we have an almost split sequence
    \[\begin{tikzcd}
    0\rar& \tau Y\rar& X\oplus E\rar& Y\rar & 0,
    \end{tikzcd}
    \]
    where $E$ is not zero and has a summand in the $\tau$-orbit of some $P(c)$ for some $c\neq a$. Then $f$ is a pushout of $\tau Y\to E$, and by induction this is injective, hence $f$ is also injective. 

    Consequently, if $X$ is a postprojective indecomposable $\K Q$-module not in the $\tau$-orbit of $P(a)$, we can find a sequence of irreducible monomorphisms from $X$ to $\tau^{-r} X$, and so $\Hom(X, \tau^{-r} X) \neq 0$, when $r\geq 0$. Hence $X \oplus \tau^{-r}X$ is not rigid for $r\geq 1$.
    
    Let now $T$ be a basic postprojective tilting $\K Q$-module, and consider which $\tau$-orbits the direct summands of $T$ occupy. By the preceding discussion, $T$ can have at most one summand in the $\tau$-orbit of $P(c)$ for any $c \neq a$. By \cref{lem:at_most_two_summands_in_orbit_of_P(a)} there can be at most two summands in the $\tau$-orbit of $P(a)$. In the case where there are two summands, then by the first part of this lemma there can be no summand in the $\tau$-orbit of $P(b)$.

    Since the number of summands $T$ has is equal to the number of vertices in $Q$, we see that $T$ must have exactly one summand in each $\tau$-orbit that is not the $\tau$-orbit of $P(a)$ or $P(b)$, and that there are either one or two summands in the $\tau$-orbit of $P(a)$. In particular, we have a summand in the $\tau$-orbit of each $P(c)$ for $c\neq b$.
    \end{proof}
\end{lemma}
    \begin{proof}[Proof of \Cref{thm:red_to_hyperbolic}]
        One may assume $Q$ to be connected, since it is simply a matter of treating each connected component otherwise.
        We proceed by induction on the number of vertices $n$ in $Q$. If $n$ equals 1, then $Q$ cannot be wild, so the statement holds vacuously. If $n=2$, then $Q$ must be hyperbolic since no proper full subquiver may be wild, and so $B$ is already hyperbolically concealed. This establishes the base case.

        To prove the inductive step, suppose that our claim holds for all quivers having fewer than $\ell$ vertices, for some $\ell > 2$. Suppose that $Q$ has $\ell$ vertices. If $Q$ is hyperbolic, we are done. So assume that $Q$ is not hyperbolic. Then we wish to show that there exists a proper wild subquiver $Q'$ and an idempotent $e$ such that $B/(e)$ is concealed of type $Q'$. This will be achieved by applying the previous lemmata.

        If $Q$ is a tree, then \cref{lem:hyperbolic_tips} tells us that we can find a tip that yields a wild subquiver $Q'$ when removed. Then, by \Cref{lem:HU89} and \cref{lem:remove_vertex_in_tau_orbit}, there is an idempotent $e$ such that $B/(e)$ is concealed of type $Q'$.

        If $Q$ is twice edge-connected and not hyperbolic, then there is a vertex $v$ such that $Q' \defeq Q \setminus \{v\}$ is wild. It is not hard to see that we may pick $v$ so that $Q'$ is connected. Again, by \Cref{lem:HU89} and \cref{lem:remove_vertex_in_tau_orbit}, there is an idempotent $e$ such that $B/(e)$ is concealed of type $Q'$.

        Lastly, if $Q$ is not a tree and not twice edge-connected, then removing a tip yields a subquiver $Q'$ which is not a tree. In this case, $Q'$ is either wild or of type $\qAtilde$. If $Q'$ is wild we may argue as before. If $Q'$ is of type $\qAtilde$, then $Q$ is of type $\qAtildetilde$. If $Q$ has fewer than 10 vertices, then $Q$ is hyperbolic and we are done. If $Q$ has 10 or more, then removing a vertex sufficiently far from the vertex of order 3 (named $b$ in \eqref{eq:Apq_spec}) yields a subquiver $Q'$ that has $\qEtildetilde_7$ as a subquiver, so it is wild. By \cref{lem:postprojective_tilting_modules_of_type_Atildetilde} and \cref{lem:remove_vertex_in_tau_orbit}, there is an idempotent $e$ such that $B/(e)$ is concealed of type $Q'$.

        This covers all the cases, and hence by induction we can always find an idempotent $e$ such that $B/(e)$ is hyperbolically concealed.
    \end{proof}

In \Cref{sec:red}, we apply \Cref{thm:red_to_hyperbolic} to prove results on $\tau$-tilting reduction of concealed $\K$-algebras.
\label{sec:3.1-concealed_typeA}

\subsection{Wild concealed algebras are not \texorpdfstring{$\mathbf{g}$}{g}-tame}
Let $\Lambda$ be a finite-dimensional $\K$-algebra.
The $\Z$-bilinear pairing
\begin{equation*}
    \begin{tikzcd}
            \mathrm{K}_0(\proj{\Lambda})\times \mathrm{K}_0(\mod{\Lambda}) \arrow[r,"{\langle - , - \rangle}"] & \Z \\
        {([P],[M])} \arrow[r,mapsto]\arrow[u,phantom, sloped,"\in"] & \mathrm{dim}(\Hom_{\Lambda}(P,M))\arrow[u,phantom, sloped,"\in"]
    \end{tikzcd}
\end{equation*}
induces an $\R$-bilinear pairing
\begin{equation*}
    \begin{tikzcd}[column sep=4em]
            \mathrm{K}_0(\proj{\Lambda})_{\R}\times \mathrm{K}_0(\mod{\Lambda})_{\R} \arrow[r,"{\langle - , - \rangle}"] & \R. 
    \end{tikzcd}
\end{equation*}
Given an element $\theta\in \mathrm{K}_0(\proj{\Lambda})_{\R}$ and $N\in\mod{\Lambda}$, we write $\theta(N)$ for $\langle \theta, [N] \rangle$.

Henceforth, we fix a hereditary $\K$-algebra $H=\K Q$ and a tilting $H$-module $T$. The associated tilted $\K$-algebra of type $Q$ is then given by $\End_{H}(T)\op$, which we denote by $B$. We get a linear isomorphism between the Grothendieck groups of $H$ and $B$. In this section, we aim to use this, in the case that $B$ is concealed, to show that $B$ is \g-tame only if $H$ is (\g-)tame.

\begin{lemma}\cite[§4.1(7)]{Rin84}\label{lem:HtoB_fan}
    Let $H$ be a hereditary $\K$-algebra, let $T$ be a tilting $H$-module and let $B = \End_H(T)\op$ be the associated tilted $\K$-algebra. Then there are isomorphisms
    \begin{equation*}
        \begin{tikzcd}
        K_0(\proj H) \arrow[r,"{(-)_{B}}"] & K_0(\proj B)\\
        \theta \arrow[r,mapsto] \arrow[u,phantom, sloped,"\in"]& \theta_B\arrow[u,phantom, sloped,"\in"]
                \end{tikzcd}
        \end{equation*}
    and
    \begin{equation*}
        \begin{tikzcd}
        K_0(\mod H) \arrow[r,"(\widehat{-})"] & K_0(\mod B)\\
        d \arrow[r,mapsto] \arrow[u,phantom, sloped,"\in"]& \widehat d \arrow[u,phantom, sloped,"\in"]
                \end{tikzcd}
        \end{equation*}
    with the following properties:
    \begin{itemize}
        \item $\theta(d) = \theta_B(\widehat d$),
        \item If $M$ is in $\Gen T$, then $\widehat {[M]} = [\Hom(T, M)]$ and $\g^M_B$ is the \g-vector of $\Hom(T, M)$.
        \item If $N$ is in $T^{\perp}$, then $\widehat {[N]} = - [\Ext(T, N)]$.
    \end{itemize}
\end{lemma}
    \begin{proof}
        By \cite[§4.1(7)]{Rin84} there is an isomorphism $K_0(\mod H) \to K_0(\mod B)$ given by $$[M] \mapsto [\Hom(T, M)] - [\Ext(T, M)].$$ We take this to be our $(\widehat -)$. Note that $\widehat{[M]} = [\Hom(T, M)]$ and $\widehat{[N]} = -[\Ext(T, N)]$ for $M \in \Gen T$ and $N \in T^\perp$. 
        Now, since $H$ and $B$ have finite global dimension, the inclusions $K_0(\proj H) \hookrightarrow K_0(\mod H)$ and $K_0(\proj B) \hookrightarrow K_0(\mod B)$ are isomorphisms. Hence we can define $(-)_B$ as the unique map making the square below commute.
        \begin{center}
        \begin{tikzcd}
            K_0(\proj H) \ar[r, "(-)_B"] \ar[d, "\cong"] & K_0(\proj B) \ar[d, "\cong"]\\
            K_0(\mod H) \ar[r, "(\widehat{-})"] & K_0(\mod B)
        \end{tikzcd}
        \end{center}
        If $M \in \Gen T$, then both $M$ and $\Hom(T, M)$ have projective dimension at most 1. Then the vertical isomorphisms identify $[M]$ and $[\Hom(T, M)]$ with their respective \g-vectors, hence $(-)_B$ maps the \g-vector of $M$ to the \g-vector of $\Hom(T, M)$. 
        Lastly, note that when $T_1$ and $T_2$ are in $\add T$, we have 
        \begin{align*}
            \g^{T_1}\left(T_2\right) = \HomDim\left(T_1, T_2\right) = \HomDim\left(\Hom\left(T, T_1\right),\Hom\left(T, T_2\right)\right) = \g^{T_1}_B\left(\widehat{T_2}\right),
        \end{align*}
        where 
        \begin{align}
            \HomDim(T_1, T_2) &\coloneqq \dim_{\K}\Hom_{H}(T_1, T_2), \label{eq:defHomDim} \\
            \HomDim\left(\Hom\left(T, T_1\right),\Hom\left(T, T_2\right)\right) &\coloneqq \dim_{\K}\Hom_{B}\left(\Hom\left(T, T_1\right),\Hom\left(T, T_2\right)\right). \nonumber
        \end{align}
        Since the summands of $T$ form a basis for $K_0(\mod H)$, the same holds for all $\theta$ and $d$ in $K_0(\proj H)$ and $K_0(\mod H)$, respectively. 
    \end{proof}

Thinking of the \g-vector fans of $H$ and $B$ inside $K_0(\proj H)$ and $K_0(\proj B)$, we would like to relate them using the isomorphism $(-)_B$. To do this we must first understand the \g-vector fan of $H$ a little better.
Because $H$ is hereditary, the isomorphism $K_0(\proj H) \to K_0(\mod H)$ is particularly nice. It associates each $H$-module with its \g-vector. This correspondence allows us to use classical tools that study the dimension vectors of quivers.

Let us now recall the Tits forms of hereditary $\K$-algebras. 
Suppose that $Q$ has $n$ vertices.
The \emph{Tits form} of $Q$ is the $\Z$-linear quadratic form defined by 
\begin{equation*}
    \begin{tikzcd}
            \Z^n \arrow[r,"q_H"] & \Z \\
         x \arrow[r,mapsto]\arrow[u,phantom, sloped,"\in"] & \sum_{i\in Q_0}x_i^2 - \sum_{i\in Q_0}d_{i,j}x_i x_j ,\arrow[u,phantom, sloped,"\in"]
    \end{tikzcd}
\end{equation*}
where $d_{i,j}$ is the number of arrows in $Q$ from the vertex $i$ to the vertex $j$. If $x$ is the dimension vector of an $H$-module $M$, one can express $q_H(x)$ by
\begin{equation}\label{eq:tits_dimvec}q_H(x)= \HomDim(M,M) - \ExtDim(M,M),\end{equation} 
where $\HomDim\coloneqq \dim_{\K} \Hom_{H}$ (see also \eqref{eq:defHomDim} above) and $\ExtDim \coloneqq \dim_{\K} \Ext^1_{H}$. 
We often think of $q_H$ as a function from $K_0(\mod H)$ to $\Z$. With this formulation, we have that $q_H([M]) = \g^M(M)$, and hence $q_H([M])$ is positive whenever $M$ is rigid. Using the isomorphism between $K_0(\proj H)$ and $K_0(\mod H)$, we may think of $q_H$ as a $\Z$-valued function on $K_0(\proj H)$. 

Suppose now that $H$ is hyperbolic. 
The hyperbolicity of $H$ famously enables the following characterization of the dimensions vectors in $\mod({H})$ in terms of the Tits form: for any integer vector with $q_H(x) \leq 1$, we have that $x$ or $-x$ is the dimension vector of an indecomposable $H$-module \cite[p.~78]{Kac83}. In particular, if $x$ is a positive vector with integer coefficients such that $q_H(x) \leq 1$, then $x$ is the dimension vector of an indecomposable $H$-module. Since $q_H(x)=1$ for all dimension vectors $x$ of an indecomposable $H$-module which is either postprojective or preinjective, we have that $x$ is the dimension vector of a regular $H$-module whenever $q_H(x)\leq 0$. 

The Tits form $q_H$ extends to a quadratic $\R$-linear form
\begin{equation*}
\begin{tikzcd}
    q_{H,\R}\colon \R^n \arrow[r] & \R.
\end{tikzcd}
\end{equation*}
We define the \textit{negative cone} of the Tits form to be the set of positive vectors in the negative locus of the Tits form, namely the set
\begin{equation*}
    C^{<0}_H \defeq \{x\in \R^n \sth q_{H,\R}(x)<0, x \text{ positive} \} \subseteq \R^n.
\end{equation*}
We have now established that when $H$ is hyperbolic, any rational vector in $C^{<0}_H$ is a scalar multiple of an indecomposable regular $H$-module. Using the isomorphism $$K_0(\proj H)_\R \cong K_0(\mod H)_\R \cong \R^n,$$ where $[P]$ is sent to its dimension vector $\underline{\dim}(P)$, we will think of $C^{<0}_H$ as a subset of $K_0(\proj H)_\R$, and relate it to the \g-vector fan.

\begin{remark}
    It is well known that a hereditary $\K$-algebra $H$ is wild if and only if $C^{<0}_H$ is nonempty. See for example \cite[p.~7]{Rin84}, \cite[Theorem~1.11]{DDPW08} or \cite[Lemma~1.2]{Kac80}.
\end{remark}

\begin{lemma}\label{lem:negative_cone_and_postprojectives}
    Let $H$ be a hyperbolic hereditary $\K$-algebra, let $N$ be an indecomposable postprojective $H$-module and let $\theta \in C^{<0}_H$ be a vector in the negative cone. Then $\theta(N) < 0$.
    \begin{proof}
        Let us first assume that $\theta$ is a rational vector. Then by the preceding discussion, we deduce that $\theta$ is a multiple of the $\g$-vector of a regular $H$-module $M$, and so we may assume that $\theta = \g^M$. Then $\theta(N) = \HomDim(M, N) - \ExtDim(N, M) = -\ExtDim(N, M) \leq 0$. Now, since $N$ is postprojective, it must be rigid, whence $\g^N(N) > 0$. Since the negative cone is an open subset of $\R^n$, we may pick a rational number $\varepsilon > 0$ such that $\theta + \varepsilon \g^N$ is still in $C^{<0}_H$. By the same argument as before, we have $(\theta + \varepsilon \g^N)(N) \leq 0$, but this implies that
        \begin{align*}
            \theta(N) \leq - \varepsilon \g^N(N) < 0.
        \end{align*}
        We conclude that $\theta(N) < 0$.
    \end{proof}
\end{lemma}

The last tool we will need is the concept of a wall. This was first utilized in the context of $\tau$-tilting by Brüstle--Smith--Treffinger \cite{BT19} and Asai \cite{Asa21}.

\begin{definition}\cite{BT19}
    For a $\K$-algebra $\Lambda$ and a nonzero $\Lambda$-module $M$ we define
    \begin{align*}
        \Theta_M \coloneqq \{\theta \in K_0(\proj \Lambda)_\R \mid \theta(M)=0,\theta(X)\geq 0 \text{ for } X \in \Gen M \}.
    \end{align*}
    We also write $\mathrm{Walls}$ for the union of all walls in $K_0(\proj \Lambda)_\R$, and $\overline{\mathrm{Walls}}$ for its closure in $K_0(\proj \Lambda)_\R$.
\end{definition}

By \cite[Theorem~3.17]{Asa21}, the complement of $\overline{\mathrm{Walls}}$ is exactly the union of the interiors of the maximal cones in the \g-vector fan, and $\mathrm{Walls}$ contains the rational points outside the \g-vector fan. This means that $\overline{\mathrm{Walls}}^{\mathrm o}$ is exactly the exterior of the \g-vector fan. When $\Lambda=H$ is hereditary, $q_H$ is non-negative on any point in the \g-vector fan, hence $C^{<0}_H$ does not intersect it. Therefore $C^{<0}_H$ is contained in $\overline{\mathrm{Walls}}$, and $C^{<0}_H \cap K_0(\proj H)_\Q$ is contained in $\mathrm{Walls}$.

\begin{theorem}\label{thm:notgtame}
    Let $Q$ be a wild acyclic quiver, let $H = \K Q$ be its path $\K$-algebra, and let $B$ be a concealed $\K$-algebra of type $Q$ with $B = \End_{H}(T)\op$ for a postprojective tilting $H$-module $T$. Then $B$ is not \g-tame.
\end{theorem}
\begin{proof}
 \Cref{thm:red_to_hyperbolic} provides an idempotent $e\in B$ such that $B/(e)$ is a concealed $\K$-algebra of hyperbolic type. By \Cref{lem:convidem}\eqref{lem.PY23.3.11}, the $\K$-algebra $B$ will not be $\mathbf{g}$-tame if $B/(e)$ is not $\mathbf{g}$-tame. We may therefore make the additional assumption that $H$ is hyperbolic, since it suffices to consider this special case to prove the general statement.

Suppose therefore that $H$ is hyperbolic.
Now consider a rational vector $\theta \in C^{<0}_H$, and let $M$ be an indecomposable $H$-module such that $\theta \in \Theta_M$. By \Cref{lem:negative_cone_and_postprojectives}, we have that $\theta(N) < 0$ for every postprojective $H$-module $N$. 
Therefore $M$ cannot be postprojective, so we must have that $M$ is in $\Gen T$. We wish to show that $\theta_B$ is contained in $\Theta_{\Hom(T, M)}$.

Since $M \in \Gen T$ we have $\theta_B(\Hom(T, M)) = \theta(M) = 0$. Combining the assertions in \Cref{setup:concealed}, one shows that any $B$-module in $\Gen \Hom_{\Lambda}(T, M)$ is of the form $\Hom_{\Lambda}(T, M') \oplus \Ext_{\Lambda}^1(T, N)$ with $M'$ in $\Gen T$ and $N$ in $T^\perp$. There is a  surjective $B$-homomorphism $\Hom_{\Lambda}(T, M)^m \twoheadrightarrow \Hom_{\Lambda}(T, M')$, which is induced by a $\Lambda$-homomorphism $M^m \xrightarrow{p} M'$. Consider the cokernel of $p$ and let $\pi \colon M' \to \mathrm{cok}(p)$ be the canonical map. Since the composition
\begin{equation}\label{eq:lesH}
\begin{tikzcd}
    \Hom_{\Lambda}(T, M)^m \ar[r, two heads, "p\circ-"] & \Hom_{\Lambda}(T, M') \ar[r, "\pi\circ-"]& \Hom_{\Lambda}(T,\mathrm{cok}(p))
\end{tikzcd}
\end{equation}
vanishes, we have that $\Hom(T, \pi) = 0$. But $\Hom_{\Lambda}(T,-)\colon \Gen(T) \to \cY_T$ is an equivalence, hence $\pi$ is $0$. This shows that $M^m \xrightarrow{p} M'$ is an epimorphism, whence $M'$ is in $\Gen M$.
By \cref{lem:HtoB_fan}, 
\begin{align*}
    \theta_B(\Hom(T, M') \oplus \Ext(T, N)) =\theta_B(\widehat{M'})+\theta_B(-\widehat{N}) =\theta(M') - \theta(N) \geq 0.
\end{align*} 
Thus, we have that $\theta_B$ is in $\Theta_{\Hom(T, M)}$.

This means that the image of $C^{<0}_H \cap K_0(\proj H)_\Q$ under $(-)_B$ does not intersect the interior of any maximal cone. Since the subset of rationals vectors is dense in $C^{<0}_H$, it does not intersect any maximal cone either. Hence, the image of $C^{<0}_H$ is an open set in the exterior of the \g-vector fan of $B$. In conclusion, we have that $B$ is not \g-tame.
\end{proof}
\label{sec:3.2-wild_not_gtame}

\subsection{\texorpdfstring{$\mathbf{g}$}{g}-tameness of incidence algebras}
We move on to classify $\mathbf{g}$-tame incidence $\K$-algebras of finite simply connected posets. We will use Leszczyński's characterization of tameness for incidence $\K$-algebras, which we restate in \Cref{thm:Les03} below. This result makes reference to the notion of universal Galois covering of a $\K$-algebra, in the sense of Martinez-Villa--de la Peña \cite{MP83}. Since the universal Galois covering of a simply connected poset is itself \cite[Theorem 4.2]{MP83}, one can prove \Cref{cor:concealed_gvecfan} below without relying on a detailed account of covering theory.  

\begin{theorem}[{\cite[Theorem~1.4]{Les03}}]\label{thm:Les03}
    Let $P$ be a finite poset. Then the incidence $\K$-algebra $\incidencealgebra{\K}{P}$ is tame if and only if its universal Galois covering does not contain a convex subalgebra $L$ which is a concealed $\K$-algebra of one of the types appearing in \cref{tab:Les frames}.
\begin{table}[h!]
    \centering
    \begin{tabularx}{.75\linewidth}{c X|c X}
         $\qAtildetilde_{p,q}$& \hspace*{\fill}\QuiverApqDraw\hspace*{\fill} & $\qDtildetilde_n$ & \hspace*{\fill}\QuiverDtildetildeDraw\hspace*{\fill} \\
         &\footnotesize  where the number of vertices is $p+q+1$, and among the arrows in the unoriented cycle, exactly $p$ arrows have clockwise orientation and exactly $q$ have counter-clockwise orientation. && \footnotesize where $4\leq n\leq 8$, and the number of vertices is $n+2$. \\[3pt]\hline
         &&&\\[-6pt]
         $\qEtildetilde_6$ & \hspace*{\fill}\QuiverEtildetildesixDraw\hspace*{\fill} & $\qEtildetilde_7$ & \hspace*{\fill}\QuiverEtildetildesevenDraw\hspace*{\fill} 
         \\[3pt] \hline
         &&& \\[-6pt]
         $\qEtildetilde_8$ & \hspace*{\fill}\QuiverEtildetildeeightDraw\hspace*{\fill} & $\qT_5$& \hspace*{\fill}\QuiverTfiveDraw\hspace*{\fill}
    \end{tabularx}
    \caption{Quivers appearing in Leszczyński's classification of wild posets. Undirected edges indicate that the arrows may go in either direction.}
    \label{tab:Les frames}
\end{table}
\end{theorem}

Since the universal Galois covering of a finite simply connected poset is just the poset itself, we deduce that a simply connected poset has a $\mathbf{g}$-tame incidence $\K$-algebra if and only if it does not contain a convex subposet whose incidence $\K$-algebra is a concealed $\K$-algebra of one of the types in Table \ref{tab:Les frames}.

\begin{corollary}\label{cor:concealed_gvecfan}
    Let $P$ be a finite simply connected poset. Then the incidence $\K$-algebra $\incidencealgebra{k}{P}$ is $\mathbf{g}$-tame if and only if it is tame.
\end{corollary}
\begin{proof}
    Plamondon and Yurikusa have shown that any tame finite-dimensional $\K$-algebra is \g-tame \cite{PY23}. We show that the converse holds for incidence $\K$-algebras of finite simply connected posets. 
    
    Suppose that the incidence $\K$-algebra $\incidencealgebra{k}{P}$ is wild.
    Since the universal Galois covering of $P$ is 
    just $P$, it follows from \Cref{thm:Les03} that $P$ must contain a concealed $\K$-algebra of wild type as a convex subalgebra.
    By \Cref{thm:notgtame}, this concealed $\K$-algebra cannot be $\mathbf{g}$-tame. Applying \Cref{lem:convidem}\eqref{lem:convidem2}, we conclude that $\incidencealgebra{k}{P}$ cannot be $\mathbf{g}$-tame.
\end{proof}

In the case that $P$ is multiply connected, we might not be able to reduce $P$ to a concealed $\K$-algebra of wild type, as achieved above. However we still formulate the following conjecture:
\begin{conjecture}\label{conj:g_tame_poset}
    Let $P$ be an arbitrary finite poset such that the incidence $\K$-algebra $\incidencealgebra{k}{P}$ is $\mathbf{g}$-tame. Then $\incidencealgebra{k}{P}$ is tame.
\end{conjecture}

One could conceivably aim to prove this conjecture by examining multiply connected posets with hyperbolic incidence $\K$-algebras. Here we give such an example, taken from the literature \cite{Les03,Sim10}. It is not clear to the authors whether this example is \g-tame.
\begin{example}
    Let $P$ be the multiply connected poset displayed below.
    
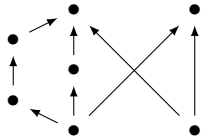
\[
    \begin{tikzpicture}[yscale=.4,xscale=.8]
        \node (MinLeft) at (0,0) [] {$\bullet$};
        \node (MinRight) at (2,0) [] {$\bullet$};
        \node (childLeft) at (-1,1) [] {$\bullet$};
        \node (childRight) at (0,2) [] {$\bullet$};
        \node (childchildLeft) at (-1,3) [] {$\bullet$};
        \node (MaxLeft) at (0,4) [] {$\bullet$};
        \node (MaxRight) at (2,4) [] {$\bullet$};

        \draw[-latex] (MinLeft)--(childLeft);
        \draw[-latex] (MinLeft)--(childRight);
        \draw[-latex] (MinLeft)--(MaxRight);
        \draw[-latex] (MinRight)--(MaxRight);
        \draw[-latex] (MinRight)--(MaxLeft);
        \draw[-latex] (childRight)--(MaxLeft);
        \draw[-latex] (childLeft)--(childchildLeft);
        \draw[-latex] (childchildLeft)--(MaxLeft);
        
    \end{tikzpicture}
    \]
    The $\K$-algebra $\incidencealgebra{\K}{P}$ is not a concealed $\K$-algebra. To see this, observe that the indecomposable $\incidencealgebra{\K}{P}$-module with dimension vector
    \[\begin{tikzpicture}[yscale=.2,xscale=0.25]
        \node (MinLeft) at (0,0) [] {$1$};
        \node (MinRight) at (1,0) [] {$1$};
        \node (childLeft) at (-1,1) [] {$0$};
        \node (childRight) at (0,2) [] {$0$};
        \node (childchildLeft) at (-1,3) [] {$0$};
        \node (MaxLeft) at (0,4) [] {$1$};
        \node (MaxRight) at (1,4) []{$1$};
    \end{tikzpicture}\]
    has projective dimension 2 and injective dimension 2 and hence by \cite[Lemma 2.1]{HRS96} the $\K$-algebra is not quasi-tilted, and in particular not concealed. Therefore, we may not apply \Cref{thm:notgtame} to decide whether this poset has a \g-tame incidence $\K$-algebra.
    The universal Galois cover of $\incidencealgebra{\K}{P}$ contains a hereditary $\K$-algebra of type $\qDtildetilde_7$ as a convex subcategory \cite[§1]{Les03}, so using the characterization of wild posets in \Cref{thm:Les03}, we observe that $\incidencealgebra{\K}{P}$ is a wild $\K$-algebra. Moreover, for any primitive idempotent $e\in \incidencealgebra{\K}{P}$, we have that $\incidencealgebra{\K}{P}/(e)$ is not wild, whence $\incidencealgebra{\K}{P}$ is hyperbolic. 
\end{example}

Another approach might be to examine the relationship between the \g-vector fan of a $\K$-algebra and the \g-vector fan of (finite quotients) of its universal Galois cover. We urge those interested to explore this further. 

We finish this section by arguing that the class of possible counterexamples to \Cref{conj:g_tame_poset} is rather constrained. Before our argument culminates in \Cref{prop:conj_lim}, we develop the necessary preparation.

Given an integer $\ell\geq 2$, let $C_\ell$ denote the following poset of type $\qAtilde_{2\ell -1}$:
    \begin{equation}\label{eq:Cell}\tag{$C_\ell$}
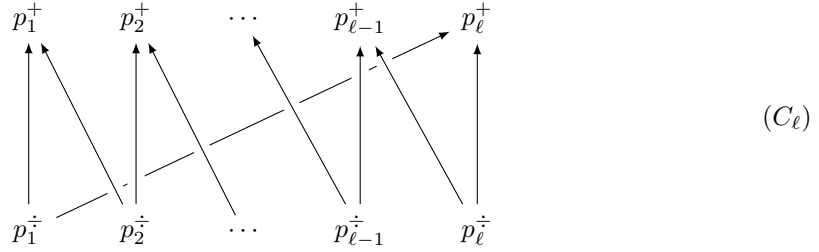
        \begin{tikzcd}[row sep = 3em, column sep=2em,every arrow/.append style={-latex}]
p^+_1                 & p^+_2                 & \cdots                                          & p^+_{\ell -1}               & p^+_\ell              \\
                                   &                                 &                                                 &                                         &                                    \\
p^\div_1 \arrow[uu,crossing over] \arrow[rrrruu,crossing over] & p^\div_2 \arrow[luu,crossing over] \arrow[uu,crossing over] & \cdots  \arrow[luu,crossing over]                                        & p^\div_{\ell -1} \arrow[uu,crossing over] \arrow[luu,crossing over] & p^\div_\ell \arrow[uu,crossing over] \arrow[luu,crossing over]
\end{tikzcd}
    \end{equation}
In the literature, such a poset is often called a \textit{crown}.
Adding a maximal element $\omega$ and a minimal element $0$ yields a poset we denote by $C_{\ell}^{\lozenge}$, as displayed below.
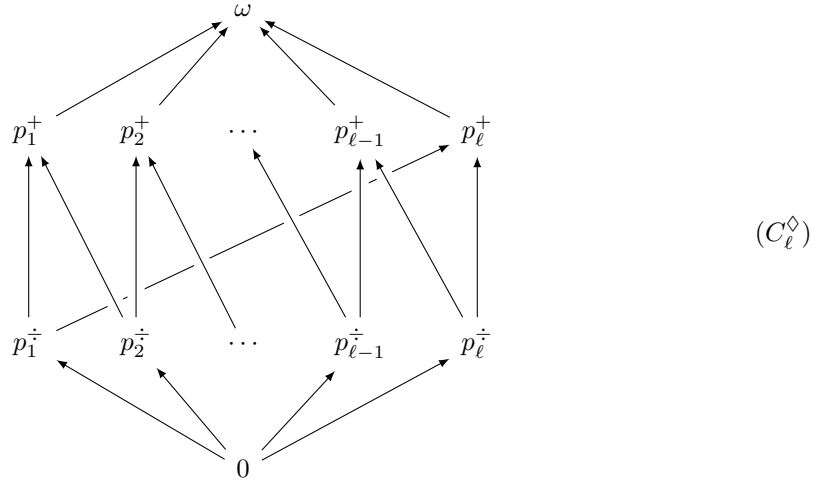
\begin{equation}\label{eq:Cell_loz}\tag{$C_\ell^{\lozenge}$}
    \begin{tikzcd}[row sep=3em,column sep=2em,every arrow/.append style={-latex}]
                                   &                                 & \omega                                               &                                         &                                    \\
p^+_1 \arrow[rru]                  & p^+_2 \arrow[ru]                & \cdots                                          & p^+_{\ell -1} \arrow[lu]                & p^+_\ell \arrow[llu]               \\
                                   &                                 &                                                 &                                         &                                    \\
p^\div_1 \arrow[uu,crossing over] \arrow[rrrruu,crossing over] & p^\div_2 \arrow[luu,crossing over] \arrow[uu,crossing over] & \cdots  \arrow[luu,crossing over]                                        & p^\div_{\ell -1} \arrow[uu,crossing over] \arrow[luu,crossing over] & p^\div_\ell \arrow[uu,crossing over] \arrow[luu,crossing over] \\
                                   &                                 & 0 \arrow[llu] \arrow[lu] \arrow[ru] \arrow[rru] &                                         &                                   
\end{tikzcd}
\end{equation}
The existence of a unique maximal element (or indeed a unique minimal element) ensures that $C_{\ell}^{\lozenge}$ is simply connected. Note that the poset $C_2^{\lozenge}$ is a subposet of the poset with the following Hasse quiver:
    \begin{equation}\label{eq:8}\tag{$\mathbf{8}$}
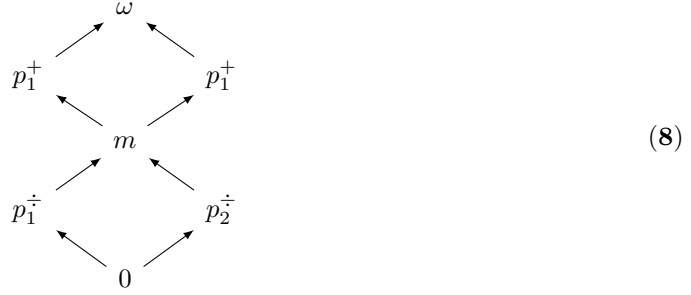
        \begin{tikzcd}[row sep=1.00em, column sep=1.75em,every arrow/.append style={-latex}]
                   & \omega                       &                    \\
p^{+}_1 \arrow[ru] &                               & p^{+}_1 \arrow[lu] \\
                   & m \arrow[lu] \arrow[ru] &                    \\
p^{\div}_1 \arrow[ru] &                               & p^{\div}_2 \arrow[lu] \\
                   & 0 \arrow[lu] \arrow[ru] &                   
\end{tikzcd}
    \end{equation}
This poset is often called the ``\textit{figure 8}'', and we denote it here by $\mathbf{8}$. Note that $\mathbf{8}$ is tame, and thus also \g-tame.

One can use the posets of the form $C_\ell^{\lozenge}$ to characterize the posets whose incidence $\K$-algebras are of global dimension 2 or greater.
Recall that the \emph{global dimension} of a finite-dimensional $\K$-algebra $\Lambda$ is defined as the supremum of the projective dimension of all finitely generated $\Lambda$-modules. The incidence $\K$-algebra of a finite poset is always of finite global dimension \cite[IX.10.3]{Mit65}, but the global dimension may depend on the characteristic of $\K$ \cite[Propostion 2.3]{IZ90}. One can nevertheless characterize the incidence $\K$-algebras of global dimension greater than 2, in a manner that is independent of the characteristic. 

\begin{theorem}[{\cite[Theorem 3.3]{IZ90}}]\label{thm:IZ90.3.3}
    Let $P$ be a finite poset. Then $\incidencealgebra{\K}{P}$ has global dimension at most 2 precisely when the following two criteria both hold.
    \begin{enumerate}[label=(IZ\arabic*)]
        \item\label{IZ1} no subposet of $P$ is isomorphic to $C_\ell^{\lozenge}$ for any $\ell\geq 3$,
        \item\label{IZ2} any subposet of $P$ which is isomorphic to $C_2^{\lozenge}$ is contained in a subposet of $P$ which is isomorphic to $\mathbf{8}$.
    \end{enumerate}
\end{theorem} 

If $P$ is a poset and $P'\subseteq P$ is a subposet, one defines the \textit{convex hull} $P'$ in $P$ as the smallest convex subposet of $P$ containing $P'$. It is denoted $\overline{P'}$. 

\begin{lemma}\label{lem:contr_wild}
Let $P$ be a finite poset.
    \begin{enumerate}
        \item\label{lem:contr_wild0.5} Let $P'$ be a subposet of $P$. If $P'$ has a unique maximal or a unique minimal element, so does its convex hull $\overline{P'}$ in $P$. In particular, the poset $\overline{P'}$ is simply connected.
        \item\label{lem:contr_wild1} Let $P'$ be a poset which is a subposet or a contraction of $P$. If $P'$ is wild, so is $P$.
        \item\label{lem:contr_wild0} For $\ell\geq 2$, incidence $\K$-algebra of the poset $C_\ell^{\lozenge}$ is representation-infinite, and wild precisely when $\ell\geq5$. 
        \item\label{lem:contr_wild3} Given an integer $\ell \geq 5$ and a poset embedding $\begin{tikzcd}C_{\ell}^{\lozenge}\arrow[r,hook] & P,\end{tikzcd}$ then the convex hull of the image of $C_{\ell}^{\lozenge}$ in $P$ is simply connected and wild. In particular, by \Cref{cor:concealed_gvecfan}, no convex hull of this form may have a \g-tame incidence $\K$-algebra.
    \end{enumerate}
\end{lemma}
\begin{proof}
    Since \eqref{lem:contr_wild0.5} immediately follows from the definitions, and \eqref{lem:contr_wild1} is known to hold \cite[Lemma 2.1]{CR19} \cite[Proposition 1.3]{Lou75},
    we first prove \eqref{lem:contr_wild0}. 
    Since $C_\ell^{\lozenge}$ contains a subposet $C_{\ell}$ of type $\qAtilde$, we deduce from \Cref{lem:Lou75} that $C_{\ell}^{\lozenge}$ has a representation-infinite incidence $\K$-algebra for all $\ell\geq 2$. 
    If $\ell\geq 5$, the subposet consisiting of the elements in $\{\omega, p^+_1, \dots , p^+_5 \}$ forms a convex subposet of type $\qT_5$, whence it follows from \Cref{lem:convidem}\eqref{lem:convidem2} and \Cref{thm:Les03} that $\incidencealgebra{\K}{C_{\ell}^{\lozenge}}$ is wild.
    On the other hand, we have that $C_4^{\lozenge}$ has a tame incidence $\K$-algebra. Suppose to the contrary that the $C_4^{\lozenge}$ has a wild incidence $\K$-algebra. By simple-connectedness and \Cref{thm:Les03}, there should exist a convex subposet of $C_4^{\lozenge}$ whose incidence $\K$-algebra is wild concealed. It is readily checked that no proper convex subposet of $C_4^{\lozenge}$ has a wild incidence $\K$-algebra, whence $\incidencealgebra{\K}{C_4^{\lozenge}}$ must be a wild concealed $\K$-algebra itself. However, since \Cref{thm:IZ90.3.3} proves that the global dimension of $\incidencealgebra{\K}{C_4^{\lozenge}}$ is larger than 2, it cannot be a concealed $\K$-algebra \cite[Lemma 2.1]{HRS96}, a contradiction, so $\incidencealgebra{\K}{C_4^{\lozenge}}$ is indeed tame. It is easy to see that $C_2^{\lozenge}$ and $C_3^{\lozenge}$ are contractions of $C_4^{\lozenge}$, whereby it now follows from \eqref{lem:contr_wild1} that $\incidencealgebra{\K}{C_2^{\lozenge}}$ and $\incidencealgebra{\K}{C_3^{\lozenge}}$ are tame as well.\footnote{We note in passing that the poset $C_3^{\lozenge}$ is exactly the poset $\totallyordered{2}\times\totallyordered{2}\times\totallyordered{2}$, see \Cref{ex:repfiniteAnTensorAm}.}

    We conclude the proof by addressing \eqref{lem:contr_wild3}. Since $C_\ell^{\lozenge}$ is wild for $\ell\geq 5$, the assertion in \eqref{lem:contr_wild3} now follows from \eqref{lem:contr_wild0.5}, \eqref{lem:contr_wild1} and \eqref{lem:contr_wild0} above.
\end{proof}

We are now in a position to argue that most wild posets do not have \g-tame incidence $\K$-algebras. In conclusion, we contend that our \Cref{conj:g_tame_poset} is justified. In an informal sense, \Cref{prop:conj_lim} shows that at most only a handful wild posets may have a \g-tame incidence $\K$-algebra, since we will be able to exclude all posets admitting a subposet which is isomorphic to $C_{\ell}^{\lozenge}$ for some $\ell\geq 5$.

\begin{proposition}\label{prop:conj_lim}
    Let $P$ be a finite multiply connected poset such that $\incidencealgebra{\K}{P}$ is both wild and \g-tame. Then $P$ must have one of the following three properties (if it exists at all):
\begin{enumerate}
    \item the global dimension of $\incidencealgebra{\K}{P}$ is exactly 2,
    \item the global dimension of $\incidencealgebra{\K}{P}$ is greater than 2 and $P$ contains a subposet of the form $C_3^{\lozenge}$ or $C_4^{\lozenge}$, so that the condition \ref{IZ1} in \Cref{thm:IZ90.3.3} is not met. 
    \item the global dimension of $\incidencealgebra{\K}{P}$ is greater than 2 and the condition \ref{IZ2} in \Cref{thm:IZ90.3.3} is not met, i.e. there is a subposet of $P$ which is isomorphic to $C_2^{\lozenge}$ and cannot be extended to a subposet of $P$ which is isomorphic to $\mathbf{8}$.
\end{enumerate}
\end{proposition}
\begin{proof}
    If the global dimension of $\incidencealgebra{\K}{P}$ is less than 2, then $\incidencealgebra{\K}{P}$ is hereditary, so it follows directly from \Cref{prop:g_tam} that $\incidencealgebra{\K}{P}$ cannot be \g-tame if it is wild, which would contradict our assumptions on $P$.
    It now suffices to show that $P$ does not have a \g-tame incidence $\K$-algebra whenever $P$ contains a subposet $P'$ of the form $C_\ell^{\lozenge}$ for some $\ell\geq 5$.
    In this case, we have that $\incidencealgebra{\K}{\overline{P'}}$ is not \g-tame by \Cref{lem:contr_wild}\eqref{lem:contr_wild3}, whence $\incidencealgebra{\K}{P}$ cannot be \g-tame, as a result of \Cref{lem:convidem}\eqref{lem:convidem2}.
\end{proof}

\appendix
\section{\texorpdfstring{$\tau$}{tau}-tilting reduction of concealed algebras}\label{sec:red}
\renewcommand{\theequation}{\Alph{section}.\alph{equation}}

This appended section builds on \Cref{sec:3.1-concealed_typeA}. We will use \Cref{thm:red_to_hyperbolic} to prove results on $\tau$-tilting reductions of concealed $\K$-algebras. For simplicity, we keep the assumption that the field $\K$ is algebraically closed, as imposed in \Cref{sec:g}.

\begin{definition}\label{def:Bongartz}
Let $\Lambda$ be a finite-dimensional $\K$-algebra and let $(M,\,R)$ be a $\tau$-rigid pair in $\mod({\Lambda})$. The \textit{Bongartz completion} of $(M,\,R)$ is the unique support $\tau$-tilting pair $(M^+,\,R)$ satisfying $\Gen{M^+}={^\perp}(\tau M)\cap R^\perp$. The \textit{Bongartz complement} of $(M,\,R)$ is the $\tau$-rigid $\Lambda$-module $M^+/M$.
\end{definition}

The Bongartz completion of a $\tau$-rigid pair always exists \cite[Theorem 2.10]{AIR14}.
The notion of Bongartz completion has its origin in the study of tilting modules over hereditary $\K$-algebras. In this setting, the Bongartz completion of a partial tilting module $U$ can be defined in terms of a universal extension of the $\K$-algebra $\Lambda$ in $\add(U)$ \cite[Lemma 2.1]{Bon81}. In the more general setting of $\tau$-tilting theory, such a universal extension does not exist, but the Bongartz completion can be recovered from a right exact sequence. Although this fact is well-known, we give a proof in \Cref{lem:BonSeq} with reference to the relevant literature.

Recall that an additive subcategory $\mathcal{X}$ of $\Modf(\Lambda)$ is said to be \textit{contravariantly finite} (resp \textit{covariantly finite}) if for every $\Lambda$-module $Y$, there exists a $\Lambda$-homomorphism $X \xrightarrow{f} Y$ (resp. $Y \xrightarrow{f} X$) such that the induced map $\Hom_{\Lambda}(X,f)$ (resp. $\Hom_{\Lambda}(f,X)$) is surjective. The $\Lambda$-homomorphism $f$ is called a \textit{right $\mathcal{X}$-approximation} (resp. \textit{left $\mathcal{X}$-approximation}) of $Y$.
A \textit{functorially finite} subcategory of $\Modf(\Lambda)$ is by definition a full subcategory that is both contravariantly and covariantly finite.

\begin{lemma}\label{lem:BonSeq}
    Let $\Lambda$ be a finite-dimensional $\K$-algebra, $(M,\,R)$ a $\tau$-rigid pair in $\mod(\Lambda)$ and $M^+/M$ the Bongartz complement of $(M,\,R)$. Then, there exists a right exact sequence 
    \begin{equation*}
    \begin{tikzcd}
        \Lambda \arrow[r,"f"]&
        L \arrow[r]& M' \arrow[r]& 0,
    \end{tikzcd}
    \end{equation*}
    where $f$ is a minimal left $\add(M^+)$-approximation, the object $L$ is in $\add(M^+)$ and $M'$ is in $\add(M)$.
\end{lemma}
\begin{proof}
Since the Bongartz completion $M^+$ is a support $\tau$-tilting $\Lambda$-module, we have that \\ $\mathcal{T}=\Gen{M^+}={^\perp}(\tau M)\cap R^\perp$ is a functorially finite torsion class of $\Modf(\Lambda)$ \cite[Theorem~2.7]{AIR14}. By the covariant finiteness of $\mathcal{T}$, we can construct the following exact sequence
\begin{equation*}
    \begin{tikzcd}
        \Lambda \arrow[r,"f"]&
        T_0 \arrow[r]& T_1 \arrow[r]& 0,
    \end{tikzcd}
    \end{equation*}
where $f$ is a minimal left $\mathcal{T}$-approximation. By \cite[Lemma~3.7]{marks2017torsion}, we know that $\add(T_0\oplus T_1)$ is the class of $\mathrm{Ext}$-projectives in $\mathcal{T}$, i.e. $\add(T_0\oplus T_1)=\{X\in \mathcal{T} \sth \Ext_\Lambda^1(X,\mathcal{T})=0\}$. Also, any $\Lambda$-module $X\in \add(T_0\oplus T_1)$ is in $\add(T_0)$ if and only if it is split projective in $\mathcal{T}$.
Now, since all summands of $M^+/M$ are split projective by \cite[Lemma~4.13]{BM21tau}, we are done. 
\end{proof}

Let $(M,R)$ be a $\tau$-rigid pair in $\Modf(\Lambda)$. The \emph{$\tau$-tilting reduction} of $\Lambda$ with respect to $(M,R)$ is defined as the $\K$-algebra $C_{(M,R)}\defeq \End_{\Lambda}({M^+}/[M])\op$, where $(M^+,R)$ is the Bongartz completion of $(M,R)$ \cite{Jas15,DIRRT23}. Let $e\in\Lambda$ be a primitive idempotent and $\Lambda e$ the corresponding indecomposable projective $\Lambda$-module. For the $\tau$-rigid pair $(\Lambda e,0)$, one can consider the $\tau$-tilting reduction $C_{(\Lambda e,0)}$, which is obviously isomorphic to the factor $\K$-algebra $\Lambda/(e)$. 

It is known that the $\tau$-tilting reduction of a hereditary $\K$-algebra is hereditary \cite[§4.16(3)]{GL91}. We include a result letting us determine the representation-type.

\begin{lemma}\label{lem:red_minwild_quiver}
    Let $Q$ be a quiver and let $(M,\,R)$ be a non-zero $\tau$-rigid pair in $\Modf(\K Q)$ in which $M$ is postprojective.
    \begin{enumerate}
        \item\label{lem:red_minwild_quiver1} Suppose that $Q$ is hyperbolic. Then the $\tau$-tilting reduction of $\K Q$ with respect to $(M,\,R)$ is tame.
        \item\label{lem:red_minwild_quiver2} Suppose that $Q$ is tame. Then the $\tau$-tilting reduction of $\K Q$ with respect to $(M,\,R)$ is representation-finite.
    \end{enumerate}
\end{lemma}
\begin{proof}
    We start by proving \eqref{lem:red_minwild_quiver1}.
    Suppose first that $R=0$.
    Let $M^+$ be the Bongartz completion of $M$. Since we are working over a hereditary $\K$-algebra, we can consider the universal extension of $\Lambda$ in $\add(M)$
    \begin{equation*}
        \begin{tikzcd}
            0 \arrow[r] & \Lambda \arrow[r] & M' \arrow[r] & M'' \arrow[r] & 0,
        \end{tikzcd}
    \end{equation*}
    and assert that $\add(M^+) = \add({M'\oplus M''})$ by \Cref{lem:BonSeq}. It follows from the extension-closure of the postprojective component that $M^+$ is a postprojective ($\tau$-)tilting $\K Q$-module. The $\tau$-tilting reduction of $\K Q$ with respect to $(M,0)$ is given by $C_M=\End_{\K Q}({M^+})\op/(e_M)$, where $e_M$ is the idempotent endomorphism corresponding to $M$. We have that $C_M$ is tame for the reason that $\End_{\K Q}(M^+)\op$ is a hyperbolic (concealed) $\K$-algebra and $C_M$ is obtained through factoring out a non-zero sum of primitive idempotents.

    In the case where $R\neq 0$, the $\tau$-tilting reduction of $\K Q$ with respect to $(M,\,R)$ is given by the $\tau$-tilting reduction of $\K Q / (e_R)$ with respect to $(M,\,0)$. The hyperbolicity of $\K Q $ makes $\K Q / (e_R)$ tame hereditary $\K$-algebra. Repeating the argument in the previous paragraph shows that the $\tau$-tilting reduction of $\K Q / (e_R)$ with respect to $(M,\,0)$ is a factor of a representation-finite or tame concealed $\K$-algebra. In particular, it is tame, and we are done.

    To prove \eqref{lem:red_minwild_quiver2}, we employ precisely the same techniques as in our proof of \eqref{lem:red_minwild_quiver1}. In the steps where we consider $\K$-algebras of the form $C/(e)$, where $e$ is an idempotent in $C$, the $\K$-algebra $C$ will be tame concealed (as opposed to hyperbolically concealed). Since tame concealed $\K$-algebras are \textit{minimally representation infinite}, in the sense that every proper quotient is representation-finite \cite[Proposition 6]{Bon17}, we may indeed repeat the proof above, \textit{mutatis mutandis}, in order to prove \eqref{lem:red_minwild_quiver2}.
\end{proof}

We note that the assumption of postprojectivity in \Cref{lem:red_minwild_quiver} is necessary for the conclusions to hold. A counterexample for \eqref{lem:red_minwild_quiver1} would be the following quiver
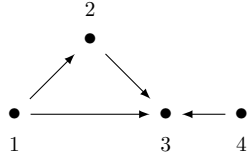
\[
\begin{tikzpicture}
    \node[] (A) at (0,0) {$\bullet$};
    \node[below=.1pt of A] [scale=.8] {$1$};
    \node[] (B) at (1,1) {$\bullet$};
    \node[above=.1pt of B] [scale=.8] {$2$};
    \node[] (C) at (2,0) {$\bullet$};
    \node[below=.1pt of C] [scale=.8] {$3$};
    \node[] (D) at (3,0) {$\bullet$};
    \node[below=.1pt of D] [scale=.8] {$4$};
    \draw[-latex] (A)--(B);
    \draw[-latex] (A)--(C);
    \draw[-latex] (B)--(C);
    \draw[-latex] (D)--(C);
\end{tikzpicture}
\]
with the $\tau$-rigid pair $(M,\,0)$, where $M$ is the indecomposable representation with dimension vector $(1,0,1,0)$. For \eqref{lem:red_minwild_quiver2}, consider the quiver obtained by removing the vertex $4$, and the indecomposable representation with dimension vector $(1,0,1)$.

\begin{lemma}\label{lem:MinimalLemIdemRed}
Let $B$ be a concealed $\K$-algebra of type $Q$. 
    \begin{enumerate}
    \item\label{lem:MinimalLemIdemRed2} If $Q$ is euclidean and $e$ is a non-zero primitive idempotent, then $B/(e)$ is representation-finite.
        \item\label{lem:MinimalLemIdemRed3} If $Q$ is hyperbolic and $e=e_1+e_2$ is a sum of distinct primitive idempotents in $B$, then $B/(e)$ is representation-finite.
    \end{enumerate}
\end{lemma}
\begin{proof}
    The assertion in \eqref{lem:MinimalLemIdemRed2} is shown by Happel--Vossieck \cite[Theorem 2]{HV83}. We deduce the assertion in \eqref{lem:MinimalLemIdemRed3} from those in \eqref{lem:MinimalLemIdemRed2} above and \cref{lem:remove_vertex_in_tau_orbit}. We first consider $B/(e_1)$. Pick a postprojective tilting $\K Q$-module $T$ such that $B = \End_{\K Q}(T)\op$. Then $e_1$ corresponds to a summand $T_1$ of $T$, and let $v$ be the vertex in $Q$ such that $T_1$ is in the $\tau$-orbit of $P(v)$. If $Q \setminus \{v\}$ is connected representation-infinite, then by \cref{lem:remove_vertex_in_tau_orbit} $B/(e_1)$ is tame concealed, and hence by \eqref{lem:MinimalLemIdemRed2} we get that $B/(e) \cong \frac{B / (e_1)}{(e_2)}$ is representation finite.

    If $Q' \defeq Q\setminus\{v\}$ is representation-finite, then following the construction in \cite{HU89}, we can construct a tilting $\K Q'$-complex $\tilde N^\bullet$, such that $\End_{\Dbmod{\K Q'}}(\tilde N^\bullet)\op \cong B/(e_1)$. Now, the $\K$-algebra $B/(e_1)$ is derived equivalent to a representation-finite hereditary $\K$-algebra, hence representation-finite, and therefore so is $B/(e)$.

    By examining the list of hyperbolic quivers \cite[Remark following Theorem~4.1]{Ker88}, we see that whenever $Q\setminus\{v\}$ is representation-infinite it is also connected, so this covers all the cases, hence $B/(e)$ is representation-finite.
\end{proof}

We are now ready to prove the main result of this appended section.

\begin{theorem}\label{prop:MinimalLemJassoRed}
    Let $Q$ be a hyperbolic quiver, and let $B=\End_{\K Q}(T)\op$ be a concealed $\K$-algebra of type $Q$.
    For $(M,\,R)\in\taurigidpair{B}$, where $M$ is postprojective, the following assertions hold:
    \begin{enumerate}
        \item\label{lem:MinimalLemJassoRed0} Let $\mathcal{Y}_T$ be the torsion class in $\Modf(B)$ as defined in \Cref{setup:concealed}. Then the Bongartz complement $M^+/M$ of $(M,\,R)$ is in $\cY_T$.
        \item\label{lem:MinimalLemJassoRed1} If $|M|+|R|\geq 1$, then the $\tau$-tilting reduction of $B$ with respect to $(M,\,R)$ is tame.
        \item\label{lem:MinimalLemJassoRed2} If $|M|+|R|\geq 2$, then the $\tau$-tilting reduction of $B$ with respect to $(M,\,R)$ is representation-finite.
    \end{enumerate}
\end{theorem}
\begin{proof}
    We first address our claim in \eqref{lem:MinimalLemJassoRed0}.
    By \Cref{lem:BonSeq}, we have a right exact sequence of the form
    \begin{equation*}
    \begin{tikzcd}
        B \arrow[r]&
        L_{\cX}\oplus L_{\cY} \arrow[r]& N \arrow[r]& 0,
    \end{tikzcd}
    \end{equation*}
    where $N\in\add(M)$ and $L_{\cX}\oplus L_{\cY}$ is a not necessarily basic object in $\add(M^+)$ with $L_{\cX}\in \cX_T$ and $L_{\cY}\in \cY_T$.
    Applying $\Tor^{B}_{1}({T},\,{-})$ gives a right exact sequence in $\K Q$
    \begin{equation*}
    \begin{tikzcd}
        0 \arrow[r]& \Tor^{B}_{1}({T},\,{L_{\cX}}) \arrow[r]& \Tor^{B}_{1}({T},\,{N}) \arrow[r]& 0,
    \end{tikzcd}
    \end{equation*}
    having used that $\Tor^{B}_{1}({T},\,{B})$ and $\Tor^{B}_{1}({T},\,{L_{\cY}})$ vanish, since both $B$ and $L_{\cY}$ are in $\cY_T$.
    Hence, we have that $\Tor^{B}_{1}({T},\,{L_{\mathcal{X}}}) \simeq \Tor^{B}_{1}({T},\,{N})$. From the Brenner--Butler Theorem (\cref{setup:concealed}), we deduce that ${L_{\cX}} \simeq {N}$. Thus, we have that $L_{\cX}\in \add(M)$. We can therefore take $L_{\cX}$ to be $0$ if we insist that the Bongartz completion be basic. This proves that $M^+/M\in\cY_T$, as claimed.

    We move on to prove \eqref{lem:MinimalLemJassoRed1} {and \eqref{lem:MinimalLemJassoRed2}}.
    Since $M$ is postprojective, all of its indecomposable direct summands are in $\cY_T$.
    By \eqref{lem:MinimalLemJassoRed0}, the Bongartz completion $M^+$ is in $\cY_T$.
    Transferring endomorphisms along the equivalence $T\otimes_{B} - \colon \cY_T \to \cT_T$ provides an isomorphism
    \begin{equation}\label{eq:EndT'}
       C_{(M,R)}\coloneqq\End_{B}(M^+)\op/[M] \simeq \End_{\K Q}\left({T\otimes_{B} M^+/M}\right)\op/[T\otimes_{B}M].
    \end{equation}
    The $\K Q$-module $T'=T\otimes_{B} {M^+_{\cY}} \in\cT_T$ is clearly partial tilting, and even tilting since the Brenner--Butler equivalence  $T\otimes_{B} - \colon \cY_T \to \cT_T$ preserves the number of indecomposable direct summands. Since $T'$ becomes the Bongartz completion of $T\otimes_{B}M$, the right hand side of \eqref{eq:EndT'} is then the $\tau$-tilting reduction of $\K Q$ with respect to $T\otimes_{B}M$, whence it follows from \Cref{lem:red_minwild_quiver}\eqref{lem:red_minwild_quiver1} that $C_{(M,R)}$ is tame.
    {Our claim in \eqref{lem:MinimalLemJassoRed2} is proved using the same reasoning as above, applying \Cref{lem:red_minwild_quiver}\eqref{lem:red_minwild_quiver2} in place of \Cref{lem:red_minwild_quiver}\eqref{lem:red_minwild_quiver1}.} 
    
\end{proof}

\bibliographystyle{amsalpha}
\bibliography{refs}

\newcommand{\etalchar}[1]{$^{#1}$}
\providecommand{\bysame}{\leavevmode\hbox to3em{\hrulefill}\thinspace}
\providecommand{\MR}{\relax\ifhmode\unskip\space\fi MR }
\providecommand{\MRhref}[2]{%
  \href{http://www.ams.org/mathscinet-getitem?mr=#1}{#2}
}
\providecommand{\href}[2]{#2}
\begin{thebibliography}{MVdlPn83}

\bibitem[Aih13]{Aihara13}
Takuma Aihara, \emph{Tilting-connected symmetric algebras}, Algebr. Represent. Theory \textbf{16} (2013), no.~3, 873--894. \MR{3049676}

\bibitem[AIR14]{AIR14}
Takahide Adachi, Osamu Iyama, and Idun Reiten, \emph{{$\tau$}-tilting theory}, Compos. Math. \textbf{150} (2014), no.~3, 415--452. \MR{3187626}

\bibitem[APR79]{APR79}
Maurice Auslander, Mar\'{\i}a~In\'{e}s Platzeck, and Idun Reiten, \emph{Coxeter functors without diagrams}, Trans. Amer. Math. Soc. \textbf{250} (1979), 1--46. \MR{530043}

\bibitem[Asa21]{Asa21}
Sota Asai, \emph{The wall-chamber structures of the real {G}rothendieck groups}, Adv. Math. \textbf{381} (2021), Paper No. 107615, 44. \MR{4205715}

\bibitem[AY23]{AY23}
Toshitaka Aoki and Toshiya Yurikusa, \emph{Complete gentle and special biserial algebras are {$g$}-tame}, J. Algebraic Combin. \textbf{57} (2023), no.~4, 1103--1137. \MR{4588126}

\bibitem[BB80]{BB80}
Sheila Brenner and M.~C.~R. Butler, \emph{Generalizations of the {B}ernstein-{G}el' fand-{P}onomarev reflection functors}, Representation theory, {II} ({P}roc. {S}econd {I}nternat. {C}onf., {C}arleton {U}niv., {O}ttawa, {O}nt., 1979), Lecture Notes in Math., vol. 832, Springer, Berlin, 1980, pp.~103--169. \MR{607151}

\bibitem[BFZ05]{BFZ05}
Arkady Berenstein, Sergey Fomin, and Andrei Zelevinsky, \emph{Cluster algebras. {III}. {U}pper bounds and double {B}ruhat cells}, Duke Math. J. \textbf{126} (2005), no.~1, 1--52. \MR{2110627}

\bibitem[BGP73]{BGP73}
I.~N. Bern\v{s}te\u{\i}n, I.~M. Gel'fand, and V.~A. Ponomarev, \emph{Coxeter functors, and {G}abriel's theorem}, Uspehi Mat. Nauk \textbf{28} (1973), no.~2(170), 19--33. \MR{393065}

\bibitem[BM21]{BM21tau}
Aslak~Bakke Buan and Bethany~Rose Marsh, \emph{{$\tau$}-exceptional sequences}, J. Algebra \textbf{585} (2021), 36--68. \MR{4272561}

\bibitem[BMR07]{BMR07}
Aslak~Bakke Buan, Bethany~R. Marsh, and Idun Reiten, \emph{Cluster-tilted algebras}, Trans. Amer. Math. Soc. \textbf{359} (2007), no.~1, 323--332. \MR{2247893}

\bibitem[Bon81]{Bon81}
Klaus Bongartz, \emph{Tilted algebras}, Representations of algebras ({P}uebla, 1980), Lecture Notes in Math., vol. 903, Springer, Berlin-New York, 1981, pp.~26--38. \MR{654701}

\bibitem[Bon17]{Bon17}
\bysame, \emph{On minimal representation-infinite algebras}, arXiv preprint arXiv:1705.10858 (2017).

\bibitem[BST19]{BT19}
Thomas Br\"{u}stle, David Smith, and Hipolito Treffinger, \emph{Wall and chamber structure for finite-dimensional algebras}, Adv. Math. \textbf{354} (2019), 106746, 31. \MR{3989130}

\bibitem[CR19]{CR19}
Fr\'ed\'eric Chapoton and Baptiste Rognerud, \emph{On the wildness of {C}ambrian lattices}, Algebr. Represent. Theory \textbf{22} (2019), no.~3, 603--614. \MR{3951763}

\bibitem[DDPW08]{DDPW08}
Bangming Deng, Jie Du, Brian Parshall, and Jianpan Wang, \emph{Finite dimensional algebras and quantum groups}, Mathematical Surveys and Monographs, vol. 150, American Mathematical Society, Providence, RI, 2008. \MR{2457938}

\bibitem[Dic66]{Dic66}
Spencer~E. Dickson, \emph{A torsion theory for {A}belian categories}, Trans. Amer. Math. Soc. \textbf{121} (1966), 223--235. \MR{191935}

\bibitem[DIJ19]{DIJ19}
Laurent Demonet, Osamu Iyama, and Gustavo Jasso, \emph{{$\tau$}-tilting finite algebras, bricks, and {$g$}-vectors}, Int. Math. Res. Not. IMRN (2019), no.~3, 852--892. \MR{3910476}

\bibitem[DIR{\etalchar{+}}23]{DIRRT23}
Laurent Demonet, Osamu Iyama, Nathan Reading, Idun Reiten, and Hugh Thomas, \emph{Lattice theory of torsion classes: beyond {$\tau$}-tilting theory}, Trans. Amer. Math. Soc. Ser. B \textbf{10} (2023), 542--612. \MR{4579952}

\bibitem[Dro74]{Dro74}
Ju.~A. Drozd, \emph{Coxeter transformations and representations of partially ordered sets}, Funkcional. Anal. i Prilo\v{z}en. \textbf{8} (1974), no.~3, 34--42. \MR{351924}

\bibitem[Dro80]{Dro80}
\bysame, \emph{Tame and wild matrix problems}, Representation theory, {II} ({P}roc. {S}econd {I}nternat. {C}onf., {C}arleton {U}niv., {O}ttawa, {O}nt., 1979), Lecture Notes in Math., vol. 832, Springer, Berlin, 1980, pp.~242--258. \MR{607157}

\bibitem[FZ02a]{FZ02i}
Sergey Fomin and Andrei Zelevinsky, \emph{Cluster algebras. {I}. {F}oundations}, J. Amer. Math. Soc. \textbf{15} (2002), no.~2, 497--529. \MR{1887642}

\bibitem[FZ02b]{FZ02ii}
\bysame, \emph{Cluster algebras {II}: Finite type classification}, arXiv preprint math/0208229 (2002).

\bibitem[FZ07]{FZ07}
\bysame, \emph{Cluster algebras. {IV}. {C}oefficients}, Compos. Math. \textbf{143} (2007), no.~1, 112--164. \MR{2295199}

\bibitem[Gab72]{Gab72}
Peter Gabriel, \emph{Unzerlegbare {D}arstellungen. {I}}, Manuscripta Math. \textbf{6} (1972), 71--103; correction, ibid. {\bf 6 (1972), 309}. \MR{332887}

\bibitem[Gab73]{Gab73}
\bysame, \emph{Indecomposable representations. {II}}, Symposia {M}athematica, {V}ol. {XI} ({C}onvegno di {A}lgebra {C}ommutativa, {INDAM}, {R}ome, 1971 \& {C}onvegno di {G}eometria, {INDAM}, {R}ome, 1972), Academic Press, London-New York, 1973, pp.~81--104. \MR{340377}

\bibitem[Gab80]{Gab80}
\bysame, \emph{Auslander-{R}eiten sequences and representation-finite algebras}, Representation theory, {I} ({P}roc. {W}orkshop, {C}arleton {U}niv., {O}ttawa, {O}nt., 1979), Lecture Notes in Math., vol. 831, Springer, Berlin, 1980, pp.~1--71. \MR{607140}

\bibitem[GL91]{GL91}
Werner Geigle and Helmut Lenzing, \emph{Perpendicular categories with applications to representations and sheaves}, J. Algebra \textbf{144} (1991), no.~2, 273--343. \MR{1140607}

\bibitem[GP72]{GP72}
I.~M. Gel'fand and V.~A. Ponomarev, \emph{Problems of linear algebra and classification of quadruples of subspaces in a finite-dimensional vector space}, Hilbert space operators and operator algebras ({P}roc. {I}nternat. {C}onf., {T}ihany, 1970), Colloq. Math. Soc. J\'anos Bolyai, vol. Vol. 5, North-Holland, Amsterdam-London, 1972, pp.~163--237. \MR{357428}

\bibitem[Hil06]{Hil06}
Lutz Hille, \emph{On the volume of a tilting module}, Abh. Math. Sem. Univ. Hamburg \textbf{76} (2006), 261--277. \MR{2293445}

\bibitem[HR82]{HR82}
Dieter Happel and Claus~Michael Ringel, \emph{Tilted algebras}, Trans. Amer. Math. Soc. \textbf{274} (1982), no.~2, 399--443. \MR{675063}

\bibitem[HRS96]{HRS96}
Dieter Happel, Idun Reiten, and Sverre~O. Smal\o, \emph{Tilting in abelian categories and quasitilted algebras}, Mem. Amer. Math. Soc. \textbf{120} (1996), no.~575, viii+ 88. \MR{1327209}

\bibitem[HU89]{HU89}
Dieter Happel and Luise Unger, \emph{Factors of wild, concealed algebras}, Math. Z. \textbf{201} (1989), no.~4, 477--483. \MR{1004168}

\bibitem[HV83]{HV83}
Dieter Happel and Dieter Vossieck, \emph{Minimal algebras of infinite representation type with preprojective component}, Manuscripta Math. \textbf{42} (1983), no.~2-3, 221--243. \MR{701205}

\bibitem[IZ90]{IZ90}
Kiyoshi Igusa and Dan Zacharia, \emph{On the cohomology of incidence algebras of partially ordered sets}, Comm. Algebra \textbf{18} (1990), no.~3, 873--887. \MR{1052771}

\bibitem[Jan57]{Jan57}
James~P. Jans, \emph{On the indecomposable representations of algebras}, Ann. of Math. (2) \textbf{66} (1957), 418--429. \MR{88485}

\bibitem[Jas15]{Jas15}
Gustavo Jasso, \emph{Reduction of {$\tau$}-tilting modules and torsion pairs}, Int. Math. Res. Not. IMRN (2015), no.~16, 7190--7237. \MR{3428959}

\bibitem[Kac80]{Kac80}
V.~G. Kac, \emph{Infinite root systems, representations of graphs and invariant theory}, Invent. Math. \textbf{56} (1980), no.~1, 57--92. \MR{557581}

\bibitem[Kac83]{Kac83}
Victor~G. Kac, \emph{Root systems, representations of quivers and invariant theory}, Invariant theory ({M}ontecatini, 1982), Lecture Notes in Math., vol. 996, Springer, Berlin, 1983, pp.~74--108. \MR{718127}

\bibitem[Ker88]{Ker88}
Otto Kerner, \emph{Preprojective components of wild tilted algebras}, Manuscripta Math. \textbf{61} (1988), no.~4, 429--445. \MR{952088}

\bibitem[Kle72]{Kle72}
M.~M. Kleiner, \emph{Partially ordered sets of finite type}, Zap. Nau\v cn. Sem. Leningrad. Otdel. Mat. Inst. Steklov. (LOMI) \textbf{28} (1972), 32--41, Investigations on the theory of representations. \MR{332585}

\bibitem[Lad08]{Lad08}
Sefi Ladkani, \emph{On derived equivalences of categories of sheaves over finite posets}, J. Pure Appl. Algebra \textbf{212} (2008), no.~2, 435--451. \MR{2357344}

\bibitem[Len99]{Len99}
Helmut Lenzing, \emph{Coxeter transformations associated with finite-dimensional algebras}, Computational methods for representations of groups and algebras ({E}ssen, 1997), Progr. Math., vol. 173, Birkh\"auser, Basel, 1999, pp.~287--308. \MR{1714618}

\bibitem[Les94]{Les94}
Zbigniew Leszczy{\'n}ski, \emph{On the representation type of tensor product algebras}, Fund. Math. \textbf{144} (1994), no.~2, 143--161. \MR{1273693}

\bibitem[Les03]{Les03}
\bysame, \emph{Representation-tame incidence algebras of finite posets}, Colloq. Math. \textbf{96} (2003), no.~2, 293--305. \MR{2010361}

\bibitem[Lou75a]{Lou75}
Mich\`ele Loupias, \emph{Indecomposable representations of finite ordered sets}, Representations of algebras ({P}roc. {I}nternat. {C}onf., {C}arleton {U}niv., {O}ttawa, {O}nt., 1974), Lecture Notes in Math., vol. Vol. 488, Springer, Berlin-New York, 1975, pp.~201--209. \MR{412210}

\bibitem[Lou75b]{Lou75fr}
\bysame, \emph{Repr\'{e}sentations ind\'{e}composables des ensembles ordonn\'{e}s finis}, S\'{e}minaire d'{A}lg\`ebre {N}on {C}ommutative (ann\'{e}e 1974/1975), Publ. Math. Orsay, No. 154-7543, Universit\'{e} de Paris XI, U.E.R. Math\'{e}matique, Orsay, 1975, pp.~Exp. No. 7, 15. \MR{424622}

\bibitem[Mit65]{Mit65}
Barry Mitchell, \emph{Theory of categories}, Pure and Applied Mathematics, vol. Vol. XVII, Academic Press, New York-London, 1965. \MR{202787}

\bibitem[Mou23]{Mou23}
Kaveh Mousavand, \emph{{$\tau$}-tilting finiteness of biserial algebras}, Algebr. Represent. Theory \textbf{26} (2023), no.~6, 2485--2522. \MR{4681324}

\bibitem[MP23]{MP23}
Kaveh Mousavand and Charles Paquette, \emph{Minimal ({$\tau $}-)tilting infinite algebras}, Nagoya Math. J. \textbf{249} (2023), 221--238. \MR{4545223}

\bibitem[M{\v{S}}17]{marks2017torsion}
Frederik Marks and Jan {\v{S}}\v{t}ov\'{\i}\v{c}ek, \emph{Torsion classes, wide subcategories and localisations}, Bull. Lond. Math. Soc. \textbf{49} (2017), no.~3, 405--416. \MR{3723626}

\bibitem[MVdlPn83]{MP83}
R.~Mart\'{\i}nez-Villa and J.~A. de~la Pe\~{n}a, \emph{The universal cover of a quiver with relations}, J. Pure Appl. Algebra \textbf{30} (1983), no.~3, 277--292. \MR{724038}

\bibitem[MW21]{MW21}
Kengo Miyamoto and Qi~Wang, \emph{On $\tau$-tilting finiteness of tensor products between simply connected algebras}, arXiv preprint arXiv:2106.06423 (2021).

\bibitem[Naz74]{Naz74}
L.~A. Nazarova, \emph{Representations of partially ordered sets of infinite type}, Funkcional. Anal. i Prilo\v{z}en. \textbf{8} (1974), no.~4, 93--94. \MR{354455}

\bibitem[NRi72]{NR72}
L.~A. Nazarova and A.~V. Ro\u~iter, \emph{Representations of partially ordered sets}, Zap. Nau\v cn. Sem. Leningrad. Otdel. Mat. Inst. Steklov. (LOMI) \textbf{28} (1972), 5--31, Investigations on the theory of representations. \MR{340121}

\bibitem[NRi75]{NR75}
\bysame, \emph{Kategorielle {M}atrizen-{P}robleme und die {B}rauer-{T}hrall-{V}ermutung}, Mitt. Math. Sem. Giessen (1975), i+153, Aus dem Russischen \"ubersetzt von K. Nikander. \MR{485998}

\bibitem[Pla19]{Pla19}
Pierre-Guy Plamondon, \emph{{$\tau$}-tilting finite gentle algebras are representation-finite}, Pacific J. Math. \textbf{302} (2019), no.~2, 709--716. \MR{4036747}

\bibitem[PYK23]{PY23}
Pierre-Guy Plamondon, Toshiya Yurikusa, and Bernhard Keller, \emph{Tame algebras have dense g-vector fans}, Int. Math. Res. Not. IMRN (2023), no.~4, 2701--2747. \MR{4565625}

\bibitem[Rin80]{Rin80}
Claus~Michael Ringel, \emph{On algorithms for solving vector space problems. {I}. {R}eport on the {B}rauer-{T}hrall conjectures: {R}ojter's theorem and the theorem of {N}azarova and {R}ojter}, Representation theory, {I} ({P}roc. {W}orkshop, {C}arleton {U}niv., {O}ttawa, {O}nt., 1979), Lecture Notes in Math., vol. 831, Springer, Berlin, 1980, pp.~104--136. \MR{607142}

\bibitem[Rin84]{Rin84}
\bysame, \emph{Tame algebras and integral quadratic forms}, Lecture Notes in Mathematics, vol. 1099, Springer-Verlag, Berlin, 1984. \MR{774589}

\bibitem[Rot64]{Rot64}
Gian-Carlo Rota, \emph{On the foundations of combinatorial theory. {I}. {T}heory of {M}\"obius functions}, Z. Wahrscheinlichkeitstheorie und Verw. Gebiete \textbf{2} (1964), 340--368. \MR{174487}

\bibitem[RV87]{RV87}
Claus~Michael Ringel and Dieter Vossieck, \emph{Hammocks}, Proc. London Math. Soc. (3) \textbf{54} (1987), no.~2, 216--246. \MR{872806}

\bibitem[Sim10]{Sim10}
Daniel Simson, \emph{Integral bilinear forms, {C}oxeter transformations and {C}oxeter polynomials of finite posets}, Linear Algebra Appl. \textbf{433} (2010), no.~4, 699--717. \MR{2654101}

\bibitem[Ung86]{Ung86}
Luise Unger, \emph{Lower bounds for faithful, preinjective modules}, Manuscripta Math. \textbf{57} (1986), no.~1, 1--31. \MR{866403}

\bibitem[Ung90]{unger_concealed_1990}
\bysame, \emph{The concealed algebras of the minimal wild, hereditary algebras}, Bayreuth. Math. Schr. (1990), no.~31, 145--154. \MR{1056151}

\bibitem[Wan22]{Wan22}
Qi~Wang, \emph{On {$\tau$}-tilting finite simply connected algebras}, Tsukuba J. Math. \textbf{46} (2022), no.~1, 1--37. \MR{4489186}

\bibitem[Yur20]{Yur20}
Toshiya Yurikusa, \emph{Density of {$g$}-vector cones from triangulated surfaces}, Int. Math. Res. Not. IMRN (2020), no.~21, 8081--8119. \MR{4176846}

\bibitem[Yur23]{Yur23}
\bysame, \emph{Acyclic cluster algebras with dense {$g$}-vector fans}, Mc{K}ay correspondence, mutation and related topics, Adv. Stud. Pure Math., vol.~88, Math. Soc. Japan, Tokyo, [2023] \copyright2023, pp.~437--459. \MR{4603594}

\bibitem[Zie95]{Zie95}
G\"unter~M. Ziegler, \emph{Lectures on polytopes}, Graduate Texts in Mathematics, vol. 152, Springer-Verlag, New York, 1995. \MR{1311028}

\end{thebibliography}

\end{document}